\documentclass[12pt]{amsart}

\usepackage{graphics}
\usepackage{amsfonts}
\usepackage{amssymb}
\usepackage{framed}
\usepackage{a4}
\usepackage{latexsym}
\usepackage{dsfont}
\usepackage{color}
\usepackage{mathrsfs} 
\usepackage{hyperref}
\usepackage{mathtools}
\usepackage{tensor}
\usepackage{ulem}

\allowdisplaybreaks

\DeclareMathAlphabet{\eusm}{U}{}{}{}  % Euler script math
\SetMathAlphabet\eusm{normal}{U}{eus}{m}{n}
\SetMathAlphabet\eusm{bold}{U}{eus}{b}{n}

\DeclareMathAlphabet{\eufrak}{U}{}{}{}  % Euler fraktur math
\SetMathAlphabet\eufrak{normal}{U}{euf}{m}{n}
\SetMathAlphabet\eufrak{bold}{U}{euf}{b}{n}

\newtheorem{theorem}{Theorem}[section]
\newtheorem{proposition}[theorem]{Proposition}
\newtheorem{lemma}[theorem]{Lemma}
\newtheorem{corollary}[theorem]{Corollary}

\theoremstyle{definition}
\newtheorem{definition}[theorem]{Definition}

\theoremstyle{remark}
\newtheorem{remark}[theorem]{Remark}

\numberwithin{equation}{section}

\newcommand{\qg}{\mathbb{G}}
\newcommand{\qk}{\mathbb{K}}
\newcommand{\qx}{\mathbb{X}}

\newcommand{\pak}{\mathscr{O}(\qk)}
\newcommand{\pa}{\mathscr{O}_\qg(\qx)}
\newcommand{\pag}{\mathscr{O}(\qg)}
\newcommand{\pquo}{\mathscr{O}_\qg(\qk \backslash \qg)}
\newcommand{\pqphi}{\mathscr{O}_\qg(\Phi \backslash \qg)}

\newcommand{\conva}{\star_{_{bi}}}

\newcommand{\convnew}{\star_{_{\ell}}}

\begin{document}

\title[Invariant Markov semigroups]{Invariant Markov semigroups on quantum homogeneous spaces}

\author{Biswarup Das}
\address{Instytut Matematyczny, Uniwersytet Wroc\l awski, pl.Grunwaldzki 2/4, 50-384 Wroc\l aw, Poland}
\email{biswarup.das@math.uni.wroc.pl}

\author{Uwe Franz}
\address{Laboratoire de math\'ematiques de Besan\c{c}on, Universit\'e de Bourgogne Franche-Comt\'e, 16, route de Gray, F-25 030 Besan\c{c}on cedex, France}
\email{uwe.franz@univ-fcomte.fr}
\urladdr{http://lmb.univ-fcomte.fr/uwe-franz}

\author{Xumin Wang}
\address{Laboratoire de math\'ematiques de Besan\c{c}on, Universit\'e de Bourgogne Franche-Comt\'e, 16, route de Gray, F-25 030 Besan\c{c}on cedex, France}
\email{xumin.wang@univ-fcomte.fr}

\begin{abstract}
Invariance properties of linear functionals and linear maps on algebras of functions on quantum homogeneous spaces are studied, in particular for the special case of expected coideal *-subalgebras. Several one-to-one correspondences between such invariant functionals are established. Adding a positivity condition, this yields one-to-one correspondences of invariant quantum Markov semigroups acting on expected coideal *-subalgebras and certain convolution semigroups of states on the underlying compact quantum group. This gives an approach to classifying invariant quantum Markov semigroups on these quantum homogeneous spaces. The generators of these semigroups are viewed as Laplace operators on these spaces.

The classical sphere $S^{N-1}$, the free sphere $S^{N-1}_+$, and the half-liberated sphere $S^{N-1}_*$ are considered as examples and the generators of Markov semigroups on these spheres a classified.  We compute spectral dimensions for the three families of spheres based on the asymptotic behaviour of the eigenvalues of their Laplace operator.
\end{abstract}

\maketitle

%\setcounter{tocdepth}{4}
%\tableofcontents

%%%%%%%%%%
\section*{Introduction}
%%%%%%%%%%

Symmetry plays an essential role in many places in mathematics and in the natural sciences. Many systems are naturally invariant under the action of some group, like time or space translations, rotations, or reflections. It is therefore of great interest to characterize and classify all invariant equations for a given group action. See for example the recent books by Ming Liao \cite{liao18} and Vladimir Dobrev \cite{dobrev1}, that study invariant Markov processes and invariant differential operators, respectively. Liao's book is motivated by probability theory, whereas Dobrev's book deals with applications to physics.

Quantum groups \cite{woronowicz80,woronowicz87} provide a generalisation of groups and can be considered as a mathematical model for quantum symmetries. Dobrev \cite{dobrev2} has also studied invariant differential operators for quantum groups. The quantum groups considered in \cite{dobrev2} are q-deformations of semi-simple Lie groups.

But there exist also interesting quantum groups that are not deformations, but rather ``liberations'' of classical groups, see, e.g., \cite{vandaele+wang,wang,banica+speicher}. These ``liberated'' quantum groups furthermore have actions on interesting ``liberated'' noncommutative spaces, see, e.g., \cite{banica+goswami10}. This provides an interesting class of examples for noncommutative geometry.

Banica and Goswami investigated how to define a Dirac operator on two of these noncommutative spaces: the free sphere $S^{N-1}_+$ and the half-liberated sphere $S^{N-1}_\times$, cf.\ \cite[Theorem 6.4]{banica+goswami10}. The action of the free or the half-liberated orthogonal group yields a natural choice for the eigenspaces, but it does not suggest how to choose the eigenvalues.

In this paper we introduce an approach for classifying invariant Markov semigroups on noncommutative spaces equipped with an action of a compact quantum group. The generators of these semigroups can be considered as natural candidates for Laplace operators. Dirac operators could be obtained via Cipriani and Sauvageot's construction \cite{cs} of a derivation from a Dirichlet form, see also \cite{cipriani+franz+kula14}. Our method generalizes the case of an action of a classical compact group on a homogeneous space presented in \cite[Chapter 3]{liao04},\cite{liao15}, \cite[Chapter 1]{liao18}. Since here we are dealing only with compact quantum groups and actions on compact quantum spaces, everything can be done on the *-algebraic level. As concrete examples we study the classical sphere $S^{N-1}$, the half-liberated sphere $S^{N-1}_*$, and the free sphere $S^{N-1}_+$.

Our approach adds a positivity condition to the invariance condition in \cite{banica+goswami10}, and leads to the formula
\[
\lambda_k= -b P_k'(1) + \int_{-1}^{1} \frac{P_k(x)-P_k(1)}{x-1}{\rm d}\nu(x),\qquad k=0,1,\ldots,
\]
for the eigenvalues of the Laplace operator on the three spheres $S^{N-1}$, $S^{N-1}_*$, $S^{N-1}_+$, see Theorem \ref{eigenvalue for the markov semigroups on spheres}. Here $b$ is a positive real number, $\nu$ is a finite positive measure on the interval $[-1,1]$, and $(P_k)_{k=0}^\infty$ is a family of orthogonal polynomials that depends on which sphere we are considering.

We define spectral dimensions the three spheres by comparing the asymptotic behaviour of the eigenvalues of their Laplace operators to the Weyl formula. More precisely, the spectral dimension is defined as the abscissa of convergence of a certain zeta function defined in terms of the eigenvalues of the Laplace operator. We find, as expected, $d_L=N-1$ for the classical sphere $S^{N-1}$. For the half-liberated sphere $S^{N-1}_*$, we get $d_L=2(N-1)$. For the free sphere $S^{N-1}_+$, we obtain
\[
d_L=\left\{\begin{array}{cl}
2 & \mbox{ if }N=2,\\
+\infty & \mbox{ if }N\ge 3.
\end{array}\right.
\]
It should be noted that, for $N=2$, the half-liberated sphere $S^1_*$ and the free sphere $S^1_+$ are isomorphic. A more detailed study of the zeta function could probably be used to introduce further interesting ``invariants'' for these noncommutative manifolds.

We now provide a brief description for the content of each section.

In Section \ref{sec-prelim}, we recall some definitions and facts about quantum group actions, quantum quotient spaces, idempotent states, and quantum Markov semigroups

Section \ref{sec-actions} gives an overview of the actions and the notions of invariance that we will consider. Proposition \ref{prop-central} shows that convolution by a central functional defines an invariant operator.

In Section \ref{sec-invfunct}, we state and prove one-to-one correspondences between various invariant linear functionals and maps on a quantum homogeneous space and on the associated compact quantum group. In the following section we use these results to characterize invariant Markov semigroups on expected right coidalgebras, cf.\ Section \ref{sec-markovsg}.

Bi-invariance leads to examples of so-called quantum hypergroups, cf.\  \cite{chap+va99}, and in Section \ref{sec-qhyper} we show that invariant Markov semigroups on expected right coidalgebras are in one-to-one correspondence with convolution semigroups of states on a quantum hypergroup that is naturally associated to the coidalgebra.

Section \ref{sec-summary} provides a short summary of our main one-to-one correspondences.

The general theory developped in Sections \ref{sec-actions}, \ref{sec-invfunct}, \ref{sec-markovsg}, and \ref{sec-qhyper}, allows us to classify the generators of invariant Markov semigroups on quantum homogeneous spaces that are associated to an idempotent state $\Phi$ on the underlying compact quantum group $\mathbb{G}$, in particular if we have a good understanding of the quantum hypergroup $\Phi\backslash \mathbb{G}/\Phi$. This is slightly more general than in the classical case, where all homogeneous spaces are of quotient type, but follows similar ideas.

In Section \ref{sec-spheres}, we apply our approach to classify invariant Markov semigroups on the classical sphere $S^{N-1}$, on the half-liberated sphere $S_*^{N-1}$, and on the free sphere $S_+^{N-1}$. In Theorem \ref{eigenvalue for the markov semigroups on spheres} we give the general form of the eigenvalues of the generators of these semigroups. In the rest of Section \ref{sec-spheres} we study in more detail the orthogonal polynomials that occur in this formula. We also show in Proposition \ref{prop-not-quotient} that the  the half-liberated sphere $S_*^{N-1}$ and the free sphere $S_+^{N-1}$ are not of quotient type. In Subsection \ref{subsec-specdim} we define a zeta function in terms of the eigenvalues of generators we classified before, determine its abscissa of convergence, and compute from this the spectral dimensions of the spheres.

We think it would be interesting to extend this study to other expected quantum homogeneous spaces, e.g., those of Banica and Speicher ``easy'' compact quantum groups \cite{banica+speicher}, where many combinatorial techniques are available for explicit calculations. And it would of course be very useful to develop methods that also apply for not necessarily expected quantum homogeneous spaces.

\subsection*{Conventions:}

We use $\otimes$ both for the tensor product of vector spaces and *-algebras, and for the minimal tensor product of C${}^*$-algebras, the meaning will be clear from the context.

%%%%%%%%%%
\section{Preliminaries}\label{sec-prelim}
%%%%%%%%%%

\subsection{Compact quantum groups}

For an introduction to the theory of compact quantum groups, see \cite{woronowicz98,notes98,timmermann}.

\subsection{Actions of compact quantum groups}
\label{Subsection: actions of compact quantum groups}

We adopt the convention that for a compact quantum group $\qg$, $C^u(\qg)$ denotes the unital C*-algebra of the universal version of $\qg$, whereas $C(\qg)$ denotes that of the reduced version. We refer the reader to \cite{decommer16} for a recent survey on actions of compact quantum groups.

\begin{definition}
A \emph{right action} $\mathbb{X}\stackrel{\alpha}{\curvearrowleft}\qg$ of a compact quantum group $\mathbb{G}$ on a compact quantum space $X$ (also called a \emph{right coaction} of $C(\qg)$ on the unital C*-algebra $C(\mathbb{X})$) is a unital *-homomorphism
\[
\alpha: C(\mathbb{X}) \rightarrow C(\mathbb{X})\otimes C(\mathbb{G})
\]
such that
\begin{itemize}
\item
the \emph{coaction property} holds: 
\[
(\alpha\otimes \mathrm{id}_{C(\mathbb{G})})\circ \alpha = (\mathrm{id}_{C(\mathbb{X})}\otimes \Delta)\circ \alpha,
\]
and
\item
the density condition (also called \emph{Podle\'{s} condition})
\[
\overline{\alpha(C(\mathbb{X}))(\mathbf{1}_{C(\mathbb{X})}\otimes C(\mathbb{G}))} = C(\mathbb{X})\otimes C(\mathbb{G})
\]
holds.
\end{itemize}
\end{definition}
Associated with every right action of a compact quantum group $\qg$ on a compact quantum space $\qx$ is the Podle\'s subalgebra or the algebraic core of $C(\qx)$, which we denote by $\pa$. We refer to \cite[pp. 25 -- 27]{decommer16} for a detailed description of the  properties of $\pa$. We collect a few facts for $\pa$:
\begin{itemize}
\item
Considering $\qg\curvearrowleft\qg$ by the coproduct, the corresponding Podle\'s subalgebra (or Peter-Weyl algebra) $\pag$ is precisely the unique, dense Hopf *-algebra $\pag$ of $\qg$, which is also commonly denoted by ${\rm Pol}(\qg)$. It is spanned by the coefficients of the finite-dimensional corepresentations of $C(\qg)$.
\item 
$\pa\subset C(\qx)$ is a dense, unital * subalgebra of $C(\qx)$ \cite[Theorem 3.16]{decommer16}.  
\item
The right coaction $\alpha:C(\qx)\longrightarrow C(\qx)\otimes C(\qg)$ restricts to a right Hopf *-coaction $\pag$ on the unital * algebra $\pa$:
\[
\alpha|_{_{\pa}}:\pa\longrightarrow \pa\otimes\pag. 
\]
\end{itemize}

An action is called \emph{embeddable}, if $\mathscr{O}_\qg(\mathbb{X})$ is isomorphic to a *-subalgebra of $\pag$, such that the action corresponds to the restriction of the coproduct, i.e., if there exists an injective unital *-homomorphism $\vartheta:\pa \to \pag$ such that $(\vartheta\otimes {\rm id})\circ \alpha|_{\pa} = \Delta\circ \vartheta$.  Such actions can be given as unital *-subalgebras which are also coideals.

\begin{definition}
A \emph{left (right, resp.) coidalgebra} of $\pag$ is  unital *-subalgebra $\mathcal{C}$ of $\pag$ such that
\[
\Delta(\mathcal{C}) \subseteq \mathcal{C}\otimes \pag, 
\qquad
(\Delta(\mathcal{C}) \subseteq \pag\otimes \mathcal{C} \mbox{ resp.}).
\]
\end{definition}

\subsection{Quantum quotient space}\label{Subsubsection: Quantum quotient space} 
Let $\qk$ be a compact quantum subgroup of $\qg$, which we will take to mean: 
\begin{itemize}
\item
$\qk$ is a compact quantum group.
\item
There exists a surjective, unital *-homomorphism $\theta: C^u(\qg)\longrightarrow C^u(\qk)$ such that 
\[
(\theta\otimes\theta)\circ\Delta_u=\Delta_{u,\qk}\circ\theta 
\]
where $\Delta_u$ is the coproduct of $C^u(\qg)$ and $\Delta_{u,\qk}$ is the coproduct of $C^u(\qk)$. 
\end{itemize}

Then the C*-algebra of the left quantum quotient of $\qg$ by $\qk$, denoted $C^u(\qk\backslash\qg)$ is defined as
\[
C^u(\qk\backslash\qg):=\{x\in C^u(\qg):~(\theta\otimes \mathrm{id})(\Delta_u(x))=1_{C^u(\qk)}\otimes x\}.
\] 
$C^u(\qk\backslash\qg)$ consists of the elements of $C^u(\qg)$ that are invariant under the left action $(\theta\otimes \mathrm{id})\circ\Delta_u:C^u(\qg) \to C^u(\qk)\otimes C^u(\qg)$ of $\qk$ on $\qg$ induced by $\theta$.

We collect a few facts about the subalgebra $C^u(\qk\backslash\qg)$ below, see also \cite{decommer16, podles95}:
\begin{itemize}
\item 
$\Delta_u(C^u(\qk\backslash\qg))\subset C^u(\qk\backslash\qg)\otimes C^u(\qg)$. Letting $\Lambda:C^u(\qg)\longrightarrow C(\qg)$ be the reducing morphism,
\[
\alpha:=( \mathrm{id}\otimes \Lambda)\circ\Delta|_{_{C^u(\qk\backslash\qg)}}:C^u(\qk\backslash\qg)\longrightarrow C^u(\qk\backslash\qg)\otimes C(\qg)
\] 
is a right action of $\qg$ on $C^u(\qk\backslash\qg)$.
\item
$\mathscr{O}_\qg(\qk\backslash\qg)\subset \pag$ and it can be easily seen that $\alpha|_{_{\mathscr{O}_\qg(\qk\backslash\qg)}}=\Delta_u|_{_{\mathscr{O}_\qg(\qk\backslash\qg)}}$. 
Thus letting $W\in M\big(C(\qg)\otimes C_0(\widehat{\qg})\big)$  be the left multiplicative unitary, it follows  that 
\[
\alpha(x)=W^*(1_\qg\otimes x)W\quad(x\in\mathscr{O}_\qg(\qk\backslash\qg)). 
\]
Denoting the norm closure of $\mathscr{O}_\qg(\qk\backslash\qg)$ in $C(\qg)$ by $C(\qk\backslash\qg)$, it follows from the above equation that 
\[
\Delta|_{_{C(\qk/\qg)}}:C(\qk/\qg)\longrightarrow C(\qk\backslash\qg)\otimes C(\qg) 
\]
is a right action of $\qg$ on $C(\qk\backslash\qg)$, which restricted to $\mathscr{O}_\qg(\qk\backslash\qg)$ is the right Hopf *-algebraic coaction $\alpha|_{_{\mathscr{O}_\qg(\qk\backslash\qg)}}:\mathscr{O}_\qg(\qk\backslash\qg)\longrightarrow \mathscr{O}_\qg(\qk\backslash\qg)\otimes \pag$.
\end{itemize}

\subsection{Idempotent states}

In this paper we will be interested in actions coming from idempotent states, as in the following theorem.

\begin{theorem}\label{thm-id}
(\cite{franz+skalski09} and \cite{franz+lee+skalski16})
Let $\mathbb{G}$ be a compact quantum group.
There is a one-to-one correspondence between the following objects:
\begin{enumerate}
\item idempotent states $\Phi$ on $\pag$;
\item idempotent states $\tilde{\Phi}$ on $C^u(\mathbb{G})$;
\item expected right (equivalently, left) coidalgebras $\mathcal{A}$ in $\pag$ (denote by $\mathcal{E}:\pag\to\mathcal{A}$ the conditional expectation);
\item expected right (equivalently, left) coidalgebras $\mathsf{A}$ in $C(\mathbb{G})$ (denote by $E:C(\mathbb{G})\to\mathsf{A}$ the conditional expectation).
\end{enumerate}
\end{theorem}

The one-to-one correspondence is given by the following relations: $\tilde{\Phi}$ is a continuous extension of $\Phi$, and $\mathcal{A} = ( \mathrm{id} \otimes\Phi )\circ\Delta(\pag)$. The C$^*$-algebra $\mathsf{A}$ is the norm closure of $\mathcal{A}$ in $C(\mathbb{G})$. On $\pag$ we can recover the idempotent state as $\Phi=\varepsilon\circ \mathcal{E}$. Moreover, each of the maps $E$ and $\mathcal{E}$ preserves the Haar state. 

We wil denote by $\mathcal{I}(\mathbb{G})$ the set of idempotent states on $C^u(\mathbb{G})$. In view of the one-to-one correspondence in Theorem \ref{thm-id}, we will denote by $\mathcal{A}_\Phi$ and $\mathsf{A}_\Phi$ the right coidalgebras associated to $\Phi\in \mathcal{I}(\mathbb{G})$, and from now on we will denote by $\mathbb{E}_r^\Phi$ the conditional expecations both onto $\mathcal{A}_\Phi$ and onto $\mathsf{A}_\Phi$ in $\pag$ or $C(\mathbb{G})$, respectively. On $\pag$ this conditional expectation can be defined by the formula $\mathbb{B}_r^\Phi=(\mathrm{id}\otimes \Phi)\circ\Delta$.
The correspondence in Theorem \ref{thm-id} preserves the natural order, i.e., we have
\[
\Phi_1\star \Phi_2 = \Phi_1 \quad \Leftrightarrow \quad \mathcal{A}_{\Phi_1}\subseteq \mathcal{A}_{\Phi_2} 
\quad \Leftrightarrow \quad \mathsf{A}_{\Phi_1}\subseteq \mathsf{A}_{\Phi_2}, 
\]
since $\mathbb{E}_r^{\Phi_1}\circ \mathbb{E}_r^{\Phi_2} = \big(\mathrm{id}\otimes (\Phi_2\star \Phi_2)\big)\circ\Delta = \mathbb{E}|_r^{\Phi_1\star\Phi_2}$.

Theorem \ref{thm-id} has recently been generalized to locally compact quantum groups, see \cite{salmi+skalski16,kasprzak+khosravi16}.

Recalling the definition of quantum quotient spaces as given in the previous subsection, it is worthwhile to note the following:
\begin{itemize}
\item
Let $h_\qk$ be the Haar state on $C^u(\qk)$. Then $\Phi_\mathbb{K}=h_\qk\circ\theta\in\mathcal{I}(\qg)$, and it follows that $C^u(\qk\backslash\qg)$ is the right coidalgebra of $C^u(\qg)$ associated with $\Phi_\mathbb{K}=h_\qk\circ\theta$. 
\item
Letting $E_{\qk\backslash\qg}:=(h_\qk\circ\theta\otimes \mathrm{id})\circ\Delta_u$ and $E_{\qg/\qk}:=( \mathrm{id}\otimes h_\qk\circ\theta)\circ\Delta_u$ (both are conditional expectations), the unital *-subalgebra 
\mbox{$E_{\qk/\qg}(\pag)\cap E_{\qg/\qk}(\pag)$} is a double coset hyper bi-algebra, as considered in \cite{franz+schurmann00}.
\end{itemize}

For our set-up, we will be mainly concerned with expected right coidalgebras of $\qg$ (we remark that analogous results hold for left coidalgebras). As pointed out above quantum quotient spaces are special cases of these. We may note that expected right coidalgebras of $\qg$ are examples of quantum homogeneous spaces, i.e. quantum spaces on which the corresponding right action of $\qg$ is ergodic \cite{podles95}.

\subsection{Convolution semigroups of states and quantum Markov semigroups on compact quantum groups}

We recall a few handy definitions and facts from \cite{cipriani+franz+kula14}.

\begin{definition}\label{def-conv-sg-G}
A \emph{convolution semigroup} on a compact quantum group $\qg$ is a family $(\lambda_t)_{t\ge 0}:\pag\to\mathbb{C}$ such that
\begin{enumerate}
\item
$\lim_{t\searrow 0} \lambda_t(a)=\lambda_0(a)$ for all $a\in\pag$
(weak continuity);
\item
$\lambda_s \star \lambda_t = \lambda_{s+t}$ for all $s,t\ge 0$ (semigroup property).
\end{enumerate}
We call $(\lambda_t)_{t\ge 0}$ a \emph{convolution semigroup of states}, if the functionals $\lambda_t$ are furthermore normalized, i.e., $\lambda_t(1)=1$, and positive, i.e., $\lambda_t(a^*a)\ge 0$ for all $a\in\pag$ and all $t\ge0$.
\end{definition}

The semigroup property implies that $\lambda_0$ is idempotent, but note that unlike \cite{cipriani+franz+kula14} we do not require $\lambda_0=\varepsilon$. The convolution semigroups on $\qg$ that we will obtain from Markov semigroups on $\qg$-spaces will in general not start with the counit.

\begin{definition}\label{def-markov-sg}
A linear operator $T:\mathsf{A}\to\mathsf{A}$ on a unital C$^*$-algebra $\mathsf{A}$ is called a \emph{quantum Markov operator}, if it is completely positive and preserves the unit of $\mathsf{A}$.

A \emph{quantum Markov semigroup} on $\mathsf{A}$ is a family $(T_t)_{t\ge0}$ of Markov operators satisfying
\begin{enumerate}
\item
$\lim_{t\searrow 0} T_t(a) = T_0(a)$ in norm for all $a\in\mathsf{A}$ (pointwise norm-continuity);
\item
$T_s\circ T_t=T_{s+t}$ for all $s,t\ge 0$ (semigroup property).
\end{enumerate}
A linear operator $T:\mathcal{A}\to\mathcal{A}$ (or a family of linear operators $(T_t:\mathcal{A}\to\mathcal{A})_{t\ge0}$, resp.) on a unital *-algebra $\mathcal{A}$ is called a quantum Markov operator (semigroup, resp.), if it is the restriction of a quantum Markov operator (semigroup, resp.) on a C$^*$-algebra $\mathsf{A}$ containing $\mathcal{A}$ that preserves $\mathcal{A}$.
\end{definition}

In \cite[Theorem 3.2]{cipriani+franz+kula14} it was shown that for a convolution semigroup of states $(\lambda_t)_{t\ge0}$ with $\lambda_0=\varepsilon$ on a compact quantum group there always exists a unique quantum Markov semigoup $(T_t)_{t\ge0}$ (with $T_0={\rm id}$) on $C(\qg)$ that acts on elements $a\in\pag$ of the Hopf *-algebra as
\[
T_t(a) = ({\rm id}\otimes \lambda_t)\circ\Delta(a).
\]
Quantum Markov semigroups coming in this way from convolution semigroups of states are characterized by the invariance property $\Delta\circ T_t = ({\rm id}\otimes T_t)\circ\Delta)$, cf.\ \cite[Theorem 3.4]{cipriani+franz+kula14}.

%%%%%%%%%%	
\section{Actions and invariances} \label{sec-actions}
%%%%%%%%%%

Let us start in the algebraic setting. A functional $\phi\in \pag'$ can act in three ways on another functional $f\in\pag'$:
\begin{eqnarray*}
L_\phi f &=& \phi\star f, \\ 
R_\phi f &=& f\star \phi, \\ 
\mathrm{Ad}_\phi f &:& \mathscr{O}(\mathbb{G})\ni a \mapsto \phi\big(a_{(1)}S(a_{(3)})\big) f(a_{(2)})\in\mathbb{C},
\end{eqnarray*}
and by duality it can also act in three ways on an element $a\in\pag$:
\begin{eqnarray*}
L^*_\phi a &=& \phi(a_{(1)}) a_{(2)}, \\
R^*_\phi a &=& \phi(a_{(2)}) a_{(1)}, \\
\mathrm{Ad}^*_\phi a &=& \phi\big(a_{(1)}S(a_{(3)})\big) a_{(2)}.
\end{eqnarray*}
It is straightforward to check that we have
\[
L_{\phi_1}\circ L_{\phi_2} = L_{\phi_1\star\phi_2}, \qquad R_{\phi_1}\circ R_{\phi_2} = R_{\phi_2\star \phi_1}
\]
and
\[
L^*_{\phi_1}\circ L^*_{\phi_2} = L^*_{\phi_2\star\phi_1}, \qquad R^*_{\phi_1}\circ R^*_{\phi_2} = R^*_{\phi_1\star \phi_2}
\]
for $\phi_1,\phi_2\in\pag'$. Furthermore,
\[
\mathrm{Ad}^*_{\phi_1}\big(\mathrm{Ad}^*_{\phi_2}(a)\big) = \phi_2(a_{(1)}) \phi_1(a_{(2)}) \phi_1\big(S(a_{(4)})\big) \phi_2\big(S(a_{(5)})\big) a_{(3)} = \mathrm{Ad}^*_{\phi_1\star\phi_2}(a)
\]
for $\phi_1,\phi_2\in\pag'$, $a\in\pag$, and
\[
\mathrm{Ad}_{\phi_1}\circ\mathrm{Ad}_{\phi_2} = \mathrm{Ad}_{\phi_1\star\phi_2}.
\]

If $\phi\in(\pag'$ is positive, then it extends to a unique positive functional on $C^u(\mathbb{G})$, cf.\ \cite[Theorem 3.3]{bedos+murphy+tuset01}. In this case its actions $L^*_\phi$ and $R^*_\phi$ on $\pag$ extend continuously to unique completely positive maps on $C(\mathbb{G})$ and $C^u(\mathbb{G})$, see, e.g., \cite[Lemma 3.4]{brannan12}. $L^*_\phi$ and $R^*_\phi$ are furthermore unital iff $\phi$ is a state, i.e., if $\phi(\mathbf{1})=1$.

\begin{definition}
For a subset $M\subseteq \pag'$ we define the spaces of \emph{left $M$-invariant}, \emph{right $M$-invariant}, and \emph{adjoint $M$-invariant} functionals and  polynomial functions as
\begin{eqnarray*}
\big(\pag'\big)^{L(M)} &=& \{f\in\pag'; \forall \phi\in M, L_\phi(f)=\phi(\mathbf{1}) f\}, \\
\big(\pag'\big)^{R(M)} &=& \{f\in\pag';\forall \phi\in M, R_\phi(f)=\phi(\mathbf{1}) f\}, \\
\big(\pag'\big)^{\mathrm{Ad}(M)} &=& \{f\in\pag';\forall \phi\in M, \mathrm{Ad}_\phi(f)=\phi(\mathbf{1}) f\}, \\
\pag^{L^*(M)} &=& \{a\in\pag;\forall \phi\in M, L^*_\phi(a)=\phi(\mathbf{1})a\}, \\
\pag^{R^*(M)} &=& \{a\in\pag;\forall \phi\in M, R^*_\phi(a)=\phi(\mathbf{1})a\}, \\
\pag^{\mathrm{Ad}^*(M)} &=& \{ a\in\pag;\forall \phi\in M, \mathrm{Ad}^*_\phi(a)=\phi(\mathbf{1})a\}.
\end{eqnarray*}
The \emph{conjugate $M$-invariant} functionals and polynomial functions are
\begin{eqnarray*}
\big(\pag'\big)^{\mathrm{Conj}(M)} &=& \{f\in\pag';\forall \phi\in M, L_\phi(f)=R_\phi(f)\}, \\
\pag^{\mathrm{Conj}^*(M)} &=& \{a\in\pag;\forall \phi\in M, L^*_\phi(a)=R^*_\phi(a)\}.
\end{eqnarray*}
For $M=\pag'$ they are also called \emph{central} functionals and polynomial functions.
\end{definition}

We also define a notion of invariance for functionals and linear operators on quantum homogeneous spaces.

\begin{definition}\label{Definition: G invariant maps}
Let $\alpha:\mathscr{O}_\qg(\mathbb{X})\to\mathscr{O}_\qg(\mathbb{X})\otimes\pag$ be a Hopf *-algebraic right action of a compact quantum group $\qg$. We say that a linear map $T:\mathscr{O}_\qg(\mathbb{X}) \to \mathscr{O}_\qg(\mathbb{X})$ is \emph{$\mathbb{G}$-invariant}, if
\[
\alpha\circ T = (T\otimes \mathrm{id})\circ \alpha.
\]
\end{definition}

Let give us a first general construction of $\mathbb{G}$-invariant operators and Markov semigroups on a homogenous space.

\begin{proposition}\label{prop-central}
Let $\alpha:\mathscr{O}_\qg(\mathbb{X})\to\mathscr{O}_\qg(X)\otimes \pag$ be a right action of a compact quantum group $\mathbb{G}$.

If $\phi:\pag\to\mathbb{C}$ is a central functional, then $T_\phi=(\mathrm{id}\otimes \phi)\circ \alpha:\mathscr{O}_\qg(\mathbb{X})\to\mathscr{O}_\qg(\mathbb{X})$ is $\mathbb{G}$-invariant.

If $(\varphi_t)_{t\ge 0}$ is a central convolution semigroup of states on $\pag$ with $\varphi_0=\varepsilon$, then $T_t = (\mathrm{id}\otimes \varphi_t)\circ \alpha$ defines a $\mathbb{G}$-invariant quantum Markov semigroup with $T_0={\rm id}$ on $\mathscr{O}_{\mathbb{G}}(\mathbb{X})$.
\end{proposition}

\begin{proof}
A functional $\phi:\pag\to\mathbb{C}$ is central iff
\[
(\mathrm{id}\otimes \phi)\circ \Delta = (\phi\otimes \mathrm{id})\circ \Delta.
\]
Therefore if $\phi$ is central, then we have
\begin{eqnarray*}
\alpha\circ T_\phi &=& (\mathrm{id}\otimes \mathrm{id}\otimes \phi) \circ (\alpha\otimes\mathrm{id}) \circ \alpha \\
&=& (\mathrm{id}\otimes\mathrm{id}\otimes \phi) \circ (\mathrm{id}\otimes \Delta) \circ \alpha \\
&=& (\mathrm{id}\otimes \phi \otimes\mathrm{id}) \circ (\mathrm{id}\otimes \Delta) \circ \alpha   \\
&=& (T_\phi\otimes\mathrm{id}) \circ \alpha,
\end{eqnarray*}
as claimed. On the algebraic core we have $T_0=({\rm id}\otimes \varepsilon)\circ \alpha={\rm id}$.

The second statement follows, since the positivity of the $\varphi_t$ implies that the $T_t$ are completely positive, $T_t(\mathbf{1})=\mathbf{1}\varphi_t(\mathbf{1})=\mathbf{1}$ for all $t\ge 0$, and
\[
T_t(x) = x_{(0)}\varphi_t(x_{(1)}) \xrightarrow{t\searrow 0} x_{(0)}\varepsilon(x_{(1)}) = x \quad\text{ in norm}
\]
for $x\in\mathscr{O}_qg(\mathbb{X})$, by continuity of the convolution semigroup $(\varphi_t)_{t\ge 0}$, see also the proof of \cite[Theorem 3.2]{cipriani+franz+kula14}. Here we used Sweedler notation $\alpha(x) = x_{(0)}\otimes x_{(1)}$ for the action.
\end{proof}

%%%%%%%%%%
\section{Invariant functionals, operators and their convolutions} \label{sec-invfunct}
%%%%%%%%%%

In this section we fix an idempotent state $\Phi\in\mathcal{I}(\qg)$ and suppose $C^r(\mathbb{G}/\Phi)$, $C^r(\Phi\backslash\mathbb{G}) $, $\mathscr{O}_{\mathbb{G}}(\mathbb{G}/ \Phi)$  and $\mathscr{O}_{\mathbb{G}}(\Phi\backslash \mathbb{G}) $ denote the respective right and left coidalgebras.
$\mathbb{E}_r^\Phi$ and $\mathbb{E}_{\ell}^\Phi$ denote respectively the conditional expectations from  $\pag$ onto $\mathscr{O}_{\mathbb{G}}(\mathbb{G}/ \Phi)$ and $\mathscr{O}_{\mathbb{G}}(\Phi\backslash \mathbb{G})$.
And we use the same notations $\mathbb{E}_r^\Phi$ and $\mathbb{E}_{\ell}^\Phi$ for the conditional expectations from $C(\qg)$ onto $C^r(\mathbb{G}/\Phi)$ and $C^r(\Phi\backslash\mathbb{G})$. We may note that the restriction of the coproduct $\Delta$ to $\mathscr{O}_{\mathbb{G}}(\Phi\backslash \mathbb{G})$ and $\mathscr{O}_{\mathbb{G}}(\mathbb{G}/ \Phi)$ are respectively left and right Hopf*-algebraic coactions of $\pag$ on  $\mathscr{O}_{\mathbb{G}}(\Phi\backslash \mathbb{G}) $ and $\mathscr{O}_{\mathbb{G}}(\mathbb{G}/ \Phi)$.
We start with two lemmas which we will be using in the sequel.

\begin{lemma}\label{Lemma: useful 1st observation}
On $\pag$ the following holds:
\begin{enumerate}
\item[(a)]
$(\mathbb{E}_{\ell}^\Phi\otimes \mathrm{id})\circ\Delta=\Delta\circ \mathbb{E}_{\ell}^\Phi~;~( \mathrm{id}\otimes \mathbb{E}_r^\Phi)\circ\Delta=\Delta\circ \mathbb{E}_r^\Phi.  
$
\item[(b)]
$(\mathbb{E}_r^\Phi\otimes \mathrm{id})\circ\Delta=( \mathrm{id}\otimes \mathbb{E}_{\ell}^\Phi)\circ\Delta$.
\end{enumerate}
\end{lemma}

\begin{proof}
The identity in (a) is actually the invariance condition for the conditional expectations, as observed in \cite{franz+lee+skalski16}.

We prove (b):

\begin{equation*}
\begin{split}
(\mathbb{E}_r^\Phi\otimes \mathrm{id})\circ\Delta&=(( \mathrm{id}\otimes  \Phi )\circ\Delta\otimes \mathrm{id})\circ\Delta\\
&=( \mathrm{id}\otimes  \Phi \otimes \mathrm{id})\circ(\Delta\otimes \mathrm{id})\circ\Delta\\
&=( \mathrm{id}\otimes  \Phi \otimes \mathrm{id})\circ( \mathrm{id}\otimes\Delta)\circ\Delta\\
&=( \mathrm{id}\otimes( \Phi \otimes \mathrm{id})\circ\Delta)\circ\Delta=( \mathrm{id}\otimes \mathbb{E}_{\ell}^\Phi)\circ\Delta.
\end{split}
\end{equation*}
\end{proof}

The following is a minor variation of the result already observed in \cite[Section 3]{franz+skalski+idempotent09}.

\begin{lemma}\label{Lemma: idempotent states commute}
Let $\Phi_1,\Phi_2$ be idempotent states on $\pag$. If $\Phi_1\ast \Phi_2=\Phi_2$, then $\Phi_2\ast \Phi_1=\Phi_2$. 
\end{lemma}

\begin{proof}
Let $S$ be the antipode of $\qg$. For any idempotent state $\phi$ on $\pag$ we have $\phi\circ S=\phi$, see \cite[Section 3, pp.~10]{franz+skalski+idempotent09}, or \cite[Proposition 4]{salmi+skalski16}, where this is shown even for locally compact quantum groups. 

This along with the identity $(S\otimes S)\circ\Delta=\Delta^{\text{op}}\circ S$ immediately implies the desired result.
\end{proof}

\subsection{Invariant functionals on expected right coidalgebras}\label{Subsection: Invariant functionals on expected right coidalgebras}

We will write $\alpha:=\Delta|_{_{\mathscr{O}_{\mathbb{G}}(\Phi\backslash \mathbb{G}) }}$.
\begin{definition}\label{Definition: h invariance and h-bi-invarince}
For $\Phi\in\mathcal{I}(\mathbb{G})$, we call a functional $f$ on $\mathscr{O}_{\mathbb{G}}(\Phi\backslash \mathbb{G}) $ \emph{$\Phi$-invariant} if $(f\otimes \Phi)\circ \alpha = f$.

We call a functional $f$ on $\pag$ \emph{$\Phi$-bi-invariant} if $f\in(\pag')^{L(\{\Phi\})}\cap(\pag')^{R(\{\Phi\})}$.
\end{definition}

\begin{theorem}\label{Theorem: from Phi invariant functional to bi Phi invariant functional and vice versa}

The following holds
\begin{enumerate}
\item[(a)]
Let $f$ be a $\Phi$-invariant functional on $\mathscr{O}_{\mathbb{G}}(\Phi\backslash \mathbb{G}) $. Then the functional defined by $\mu:=f\circ \mathbb{E}_{\ell}^\Phi$ is the unique $\Phi$-bi-invariant functional on $\pag$, whose restriction to $\mathscr{O}_{\mathbb{G}}(\Phi\backslash \mathbb{G})$ is $f$.
\item[(b)]
Let $\mu$ be a $\Phi$-bi-invariant functional on $\pag$. Then $f:=\mu|_{_{\mathscr{O}_{\mathbb{G}}(\Phi\backslash \mathbb{G}) }}$ is the unique $\Phi$-invariant functional on $\mathscr{O}_{\mathbb{G}}(\Phi\backslash \mathbb{G}) $, such that $f\circ \mathbb{E}_{\ell}^\Phi=\mu$.
\end{enumerate}
\end{theorem}

\begin{proof}
We prove (a):

Let $x\in\pag$. We prove the $\Phi$-bi-invariance of $\mu$ as follows: 

Left $\Phi$-invariance:
\begin{equation*}
\begin{split}
(\Phi\otimes\mu)(\Delta(x))&=( \Phi \otimes f\circ \mathbb{E}_{\ell}^\Phi)(\Delta(x))\\
&=( \Phi \otimes f)((\mathbb{E}_r^\Phi\otimes \mathrm{id})(\Delta(x)))\quad(\text{by Lemma}~ \ref{Lemma: useful 1st observation}-(b))\\
&=( \Phi \otimes  \Phi \otimes f)(\Delta\otimes \mathrm{id})(\Delta(x))\\
&=( \Phi \ast  \Phi \otimes f)(\Delta(x))\\
&=( \Phi \otimes f)(\Delta(x))=f(\mathbb{E}_{\ell}^\Phi(x))\quad(\text{since}~ \Phi \in\mathcal{I}(\qg))\\
&=\mu(x).
\end{split}
\end{equation*}

Right $\Phi$-invariance:

\begin{equation*}
\begin{split}
(\mu\otimes  \Phi )(\Delta(x))&=(f\circ \mathbb{E}_{\ell}^\Phi\otimes  \Phi )(\Delta(x))\\
&=(f\otimes  \Phi )(\mathbb{E}_{\ell}^\Phi\otimes \mathrm{id})(\Delta(x))\\
&=(f\otimes  \Phi )\Delta(\mathbb{E}_{\ell}^\Phi(x))\quad(\text{by Lemma \ref{Lemma: useful 1st observation}}-(a))\\
&=f(\mathbb{E}_{\ell}^\Phi(x))\quad(\text{using $\Phi$-invariance of $f$})\\
&=\mu(x).
\end{split}
\end{equation*}

Now, let $\nu$ be any $\Phi$-bi-invariant functional on $\pag$ such that $\nu|_{_{\mathscr{O}_{\mathbb{G}}(\Phi\backslash \mathbb{G}) }}=f$. Then using the right $\Phi$-invariance of $\nu$ we have

\begin{equation*}
\begin{split}
\nu(x)&=( \Phi \otimes \nu)(\Delta(x))\\
&=\nu(\mathbb{E}_{\ell}^\Phi(x))\\
&=f(\mathbb{E}_{\ell}^\Phi(x))=\mu(x),
\end{split}
\end{equation*}
which proves the uniqueness.
		
(b) follows by observing that the $\Phi$-invariance of $f$ as a functional on $\mathscr{O}_{\mathbb{G}}(\Phi\backslash \mathbb{G}) $ is a consequence of the left $\Phi$-invariance of $\mu$ as a functional on $\pag$, and uniqueness can be seen easily.
\end{proof}

Thus we have a one-to-one correspondence between the set of $\Phi$-invariant functional on $\mathscr{O}_{\mathbb{G}}(\Phi\backslash \mathbb{G}) $ and that of $\Phi$-bi-invariant functionals on $\pag$. In particular, we have 

\begin{corollary}\label{Corollary: from Phi invariant state to h bi ivariant state}
There exists a one-to-one correspondence between the set of $\Phi$-invariant states on $\mathscr{O}_{\mathbb{G}}(\Phi\backslash \mathbb{G}) $ and the set of $\Phi$-bi-invariant states on $\pag$. 
\end{corollary}

\begin{proof}
This is clear, because $\mathbb{E}_{\ell}^\Phi$ is completely positive.
\end{proof}

\begin{remark}\label{Remark: functionls on coidalgebra extends to left h-invariant functional on G}

We may note that given any functional $\nu:\mathscr{O}_{\mathbb{G}}(\Phi\backslash \mathbb{G}) \longrightarrow\mathbb{C}$ the functional $\mu:\pag\longrightarrow\mathbb{C}$ defined by $\mu:=\nu\circ \mathbb{E}_{\ell}^\Phi$ is a left $\Phi$-invariant functional on $\pag$. This follows from the computations proving left $\Phi$-invariance of $\mu$ in the proof of Theorem \ref{Theorem: from Phi invariant functional to bi Phi invariant functional and vice versa}. 
\end{remark}

\subsubsection{The case $\Phi=h_\qk\circ\theta$ of quantum quotient spaces}\label{Subsubsection: invariance in quantum quotient space}

Let $\qk$ be a compact quantum subgroup of $\qg$. Let $\theta$ be the associated surjective quantum group morphism. Then $\theta:\pag\longrightarrow\pak$ is a surjective Hopf *-morphism such that $(\theta\otimes\theta)\circ\Delta=\Delta^\prime\circ\theta$ where $\Delta^\prime$ is the coproduct of $\pak$. It can be easily observed that $\beta_r:=( \mathrm{id}\otimes\theta)\circ\Delta:\pag\longrightarrow \pag\otimes\pak$ is a right Hopf *-algebraic coaction of $\pak$ on the unital *-algebra $\pag$, and similarly $\beta_l:=(\theta\otimes \mathrm{id})\circ\Delta:\pag\longrightarrow \pak\otimes\pag$ is a left Hopf *-algebraic coaction of $\pak$ on $\pag$. 

\begin{definition}\label{Definition: definition of K invariance}
We call a functional $f$ on $\pag$ \emph{$\qk$-bi-invariant} if \[(f\otimes \mathrm{id})\circ\beta_r=f(\cdot)\mathbf{1}_\mathbb{K} = (\mathrm{id}\otimes f)\circ\beta_l.\] 
\end{definition}

Let $h_\qk$ be the Haar state of $\qk$, so that $h:=h_\qk\circ\theta\in\mathcal{I}(\qg)$.

Let $\mathrm{Irr}(\qk)$ be the set of inequivalent, irreducible unitary representations of $\qk$. For $\pi\in \mathrm{Irr}(\qg)$, denote the carrier Hilbert space of $\pi$ by $H_\pi$, and let $\delta_\pi:H_\pi\longrightarrow H_\pi\otimes \pak$ be the $\pak$-comodule induced by $\pi$, as in \cite[Theorem 1.2, Lemma 1.7]{decommer16}.
Then it follows that there exists an orthonormal basis $\{e_1,\ldots,e_{\text{dim}~\pi}\}$ of $H_\pi$ such that
$\delta_\pi(e_i)=\sum_{k=1}^{\text{dim}~\pi}e_{k}\otimes\pi_{kj}$, cf.\ \cite[Theorem 1.2, Lemma 1.5]{decommer16}. For $\pi\in \mathrm{Irr}(\qk)$, let $C(\qg)_\pi:=\text{Lin}\{T\xi:~\xi\in H_\pi,T\in \text{Mor}(\pi,\beta_r)\}$, as in \cite[Definition 3.13]{decommer16}. Then it follows from \cite[Theorem 3.16]{decommer16} and its proof that $\pag=\bigoplus_{_{\pi\in {\mathrm Irr}(\qk)}}C(\qg)_\pi$.

\begin{theorem}\label{Theorem: Necessary and sufficient condition for left K invariance}
A functional $f$ on $\pag$ is left $\qk$-invariant, i.e. $(f\otimes \mathrm{id})\circ\beta_r=f$ if and only if  for those $\pi\in{\rm Irr}(\qk)$ which are inequivalent to the trivial representation, we have $f|_{_{C(\qg)_\pi}}=0$. 
\end{theorem}

\begin{proof}
Let $f$ be left $\qk$-invariant. This implies that $(f\otimes \mathrm{id})(\beta_r(T\xi))=f(T\xi)$ for all $\xi\in H_\pi$, for all $T\in \text{Mor}(\pi,\beta_r)$. Thus in particular we have
\[
f(T(e_{j}))\mathbf{1}_\mathbb{K} =(f\otimes \mathrm{id})(\beta_r(T(e_{j})))=(f\circ T\otimes \mathrm{id})\delta_\pi(e_{j}), 
\]
from which it follows that 
\[
f(T(e_{j}))\mathbf{1}_\mathbb{K}=\sum_{k=1}^{\text{dim}\pi}f(T(e_{k}))\pi_{kj} 
\]
i.e. 
\[
\sum_{k=1}^{\text{dim}\pi}f(T(e_{k}))\pi_{kj}-f(T(e_{j}))\mathbf{1}_\qk=0.
\]
Since $\pi$ is different from the trivial representation, this means, using the linear independence of the set $\{e_{j}:~j=1,2,\cdots\text{dim}~\pi\}\cup\{\mathbf{1}_\qk\}$, $f(T(e_{j}))=0$ for all $i,j$, which implies that $f|_{_{C(\qg)_\pi}}=0$.
		
Conversely suppose $f|_{_{C(\qg)_\pi}}=0$ for all those $\pi$ different from the trivial representation.
Let $x\in \pag$. Since $\pag=\bigoplus_{\pi\in Irr(\qk)}C(\qg)_\pi$ and $\beta_r(C(\qg)_\pi)\subset C(\qg)_\pi\otimes \pak$, this implies that $(f\otimes \mathrm{id})\beta_r(x)=0=f(x)$ if $x\in C(\qg)_\pi$ for $\pi\neq \mathbf{1}1_\qk$. If $x\in C(\qg)_{1_\qk}$, then $x$ is a fixed point of the coaction $\beta_r$. Thus $(f\otimes \mathrm{id})\beta_r(x)=f(x)\mathbf{1}_\qk$. Thus $f$ is left $\qk$-invariant.
\end{proof}

It is now easy to also prove a corresponding right $\qk$-invariance version of Theorem \ref{Theorem: Necessary and sufficient condition for left K invariance}:

\begin{corollary}\label{Corollary: A necessary and sufficient condition for the right K invariance of f}
Let $f$ be a functional on $\pag$. Then $f$ is right $\qk$-invariant i.e. $( \mathrm{id}\otimes f)\circ\beta_l=f$, if and only if $f|_{_{C(\qg)_\pi}}=0$ for all those $\pi\neq1_\qk$, where now for $\pi\in \mathrm{Irr}(\qk)$, $C(\qg)_\pi:=\{T\xi:~\xi\in H_\pi,T\in \text{Mor}(\pi,\beta_l)\}$. 
\end{corollary}

We now prove the main results for this subsection.
Note that $\Phi_\qk = h_\qk\circ\theta\in\mathcal{I}(\qg)$.

\begin{theorem}\label{Theorem: From bi K invariance to bi h invariance and vice versa}
A functional $f$ on $\pag$ is $\qk$-bi-invariant if and only if it is $\Phi_\qk$-bi-invariant. 
\end{theorem}

\begin{proof}
If $f$ is $\qk$-bi-invariant, it easily follows that $f$ is also $\Phi_\qk$-bi-invariant. 
		
We prove the converse implication. We prove only the left $\qk$-invariance of $f$. The proof of the right $\qk$-invariance is identical, with $\beta_r$ replaced by $\beta_l$.
		
We will use the notations in the proof of Theorem \ref{Theorem: Necessary and sufficient condition for left K invariance}. Recall that
\[
\pag=\bigoplus_{\pi\in Irr(\qk)}C(\qg)_\pi.
\]
Let $x\in C(\qg)_\pi$, such that $\pi\neq\mathbf{1}_\qk$. We may note that $( \mathrm{id}\otimes h_\qk\circ\theta)\circ\Delta$ is the conditioal expectation onto the fixed point subalgebra of the right coaction $\beta_r$. Since $(f\otimes h_\qk\circ\theta)(\Delta(x))=f(x)$ and $( \mathrm{id}\otimes h_\qk\circ\theta)(\Delta(x))=0$, this implies that $f(x)=0$. Left $\qk$-invariance of $f$ now follows from Theorem \ref{Theorem: Necessary and sufficient condition for left K invariance}.
\end{proof}

Let us recall the construction of the quantum quotient space $\qk\backslash\qg$ as explained in Subsection \ref{Subsubsection: Quantum quotient space}. As before, let us denote the Podle\'s algebra for the right action of $\qg$ on $\qk\backslash\qg$ by $\mathscr{O}_\qg(\qk\backslash\qg)$, and the corresponding right Hopf *-coaction of $\pag$ on $\mathscr{O}_\qg(\qk\backslash\qg)$ by $\alpha$.

\begin{definition}
A functional $f$ on $\pquo$ is called \emph{$\qk$-invariant} if
\[
(\theta\otimes f)\circ\alpha=f(\cdot)\mathbf{1}_\qk.
\] 
\end{definition}

\begin{remark}
We may note that the above definition of $\qk$-invariance of a functional on $\pquo$ reduces to the usual definition of $K$-invariant measure on quotient spaces $K\backslash G$ when $G$ is a classical compact group and $K$ is a compact subgroup, as introduced in \cite{liao04,liao15}. 
\end{remark}

Let us also recall from Subsection \ref{Subsubsection: Quantum quotient space} that $\pquo$ can equivalently be thought of as the right coidalgebra corresponding to the idempotent state $ h_\qk\circ\theta$ on $\pag$. Let $\mathbb{E}_{\qk\backslash\qg}:=( \mathrm{id}\otimes h_\qk\circ\theta)\circ\Delta$ be the corresponding conditional expectation associated with the idempotent state $\Phi_\qk= h_\qk\circ\theta$.

\begin{theorem}\label{Theorem: From K invariance to K-bi-invariance}
Let $f$ be a $\qk$-invariant functional on $\pquo$. Then there exists a unique $\qk$-bi-invariant functional $\mu$ on $\pag$ such that $f\circ \mathbb{E}_{\qk\backslash\qg}=\mu$. 
\end{theorem}

\begin{proof}
Since $f$ is a $\qk$-invariant functional on $\pquo$, this implies that $f$ is a $ h_\qk\circ\theta$-invariant functional in the sense of Definition \ref{Definition: h invariance and h-bi-invarince}. Thus by Theorem \ref{Theorem: from Phi invariant functional to bi Phi invariant functional and vice versa}, there exists a unique $ h_\qk\circ\theta$-bi-invariant functional $\mu$ on $\pag$ satisfying $f\circ \mathbb{E}_{\qk\backslash\qg}=\mu$. Now from Theorem \ref{Theorem: Necessary and sufficient condition for left K invariance} and Corollary \ref{Corollary: A necessary and sufficient condition for the right K invariance of f} it follows that $\mu$ is also $\qk$-bi-invariant as a functional on $\pag$. This proves the result.  
\end{proof}

As a consequence, we have the following:

\begin{corollary}\label{Corollary: From bi K invariance to K invariance}
Let $\mu$ be a $\qk$-bi-invariant functional on $\pag$. Then  \mbox{$f:=\mu|_{\pquo}$} is the unique $\qk$-invariant functional on $\pquo$ such that $f\circ \mathbb{E}_{\qk\backslash\qg}=\mu$.
\end{corollary}

\begin{proof}
We may note that the $\qk$-invariance of the functional $f$ on $\pquo$ follows from the left $\qk$-invariance of $\mu$ as a functional on $\pag$. The rest of the proof is an adaptation of the proof of Theorem \ref{Theorem: From K invariance to K-bi-invariance}. 
\end{proof}

Thus we have a one-one correspondence between $\qk$-invariant functionals on $\pquo$ and $\qk$-bi-invariant functionals on $\pag$.  This correspondence can be seen to extend the already known one-one correspondence between $K$-invariant measures on the quotient space $K\backslash G$ and $K$-bi-invariant measure on $G$ for a classical compact group $G$ and its compact subgroup $K$ \cite{liao04}.

\subsection{Convolution of functionals and invariant operators on expected right coidalgebras}

Let $\mathscr{O}_{\mathbb{G}}(\Phi\backslash \mathbb{G}) $ be an expected right coidalgebra and \mbox{$h\in\mathcal{I}(\qg)$} be the associated idempotent state. Let $\varepsilon$ denote the counit of $\qg$ and \mbox{$\alpha:=\Delta|_{_{\mathscr{O}_{\mathbb{G}}(\Phi\backslash \mathbb{G}) }}$}.

\subsubsection{Convolution of functionals on expected right coidalgebras}

\begin{definition}\label{Definition: definition for convolution of functionals on coidalgebras}
Let $f$ and $g$ be two functionals on the expected right coidalgebra $\mathscr{O}_{\mathbb{G}}(\Phi\backslash \mathbb{G}) $. We define convolution of $f$ and $g$, denoted $f\convnew g$ as the following functional on $\mathscr{O}_{\mathbb{G}}(\Phi\backslash \mathbb{G}) $:
\[
f\convnew g:=\big(f\otimes (g\circ \mathbb{E}_{\ell}^\Phi)\big)\circ\alpha. 
\]
\end{definition}

\begin{remark}
Let us make a remark on the notations used here:

For two functionals $\mu$ and $\nu$ on $\pag$, $\mu\ast\nu$ will denote the convolution defined by $\mu\ast\nu:=(\mu\otimes\nu)\circ\Delta$, whereas for two functionals $f$ and $g$ on $\mathscr{O}_{\mathbb{G}}(\Phi\backslash \mathbb{G}) $, $f\convnew g$ will denote the functional on $\mathscr{O}_{\mathbb{G}}(\Phi\backslash \mathbb{G}) $, as given in Definition \ref{Definition: definition for convolution of functionals on coidalgebras}.
\end{remark}

\begin{theorem}\label{Theorem: Convolution at the group level restricts to give convolution at the coidalgebra level}
Let $f_1$ and $f_2$ be two $\Phi$-invariant functionals on $\mathscr{O}_{\mathbb{G}}(\Phi\backslash \mathbb{G}) $ and $\mu_1$ and $\mu_2$ be their unique $\Phi$-bi-invariant extensions to $\pag$, as given by Theorem \ref{Theorem: from Phi invariant functional to bi Phi invariant functional and vice versa}.

Then the following holds:
\begin{enumerate}
\item[(a)]
$f_1\convnew f_2$ is a $\Phi$-invariant functional on $\mathscr{O}_{\mathbb{G}}(\Phi\backslash \mathbb{G}) $.
\item[(b)]
$\mu_1\ast\mu_2$ is the unique $\Phi$-bi-invariant extension of $f_1\convnew f_2$ to $\pag$.
\end{enumerate}
\end{theorem}

\begin{proof}
We prove (a):

\begin{equation*}
\begin{split}
(f_1\convnew f_2\otimes \Phi)\circ\alpha&=(f_1\otimes f_2\circ \mathbb{E}_{\ell}^\Phi\otimes \Phi)\circ(\alpha\otimes \mathrm{id})\circ\alpha\\
&=(f_1\otimes f_2\circ \mathbb{E}_{\ell}^\Phi\otimes \Phi)\circ( \mathrm{id}\otimes\Delta)\circ\alpha\\
&=(f_1\otimes (f_2\circ \mathbb{E}_{\ell}^\Phi\otimes \Phi)\circ\Delta)\circ\alpha\\
&=(f_1\otimes(\mu_2\otimes \Phi)\circ\Delta)\circ\alpha\\
&=(f_1\otimes\mu_2)\circ\alpha\quad(\text{using left $\Phi$-invariance of $\mu_2$})\\
&=(f_1\otimes f_2\circ \mathbb{E}_{\ell}^\Phi)\circ\alpha=f_1\convnew f_2.
\end{split}
\end{equation*}
To prove (b): 
		
Using the fact that both $\mu_1$ and $\mu_2$ are $\Phi$-bi-invariant functionals on $\pag$, it is easy to see that $\mu_1\ast\mu_2$ is a $\Phi$-bi-invariant functional on $\pag$.

Let $x\in \mathscr{O}_{\mathbb{G}}(\Phi\backslash \mathbb{G}) $.
\begin{equation*}
\begin{split}
(\mu_1\ast\mu_2)(x)&=(\mu_1\ast\mu_2)(\mathbb{E}_{\ell}^\Phi(x))\\
&=(f_1\circ \mathbb{E}_{\ell}^\Phi\otimes f_2\circ \mathbb{E}_{\ell}^\Phi)(\Delta(\mathbb{E}_{\ell}^\Phi(x)))\\
&=(f_1\circ (\mathbb{E}_{\ell}^\Phi)^2\otimes f_2\circ \mathbb{E}_{\ell}^\Phi)(\Delta(x))\quad(\text{using the invariance of $\mathbb{E}_{\ell}^\Phi$})\\
&=(f_1\circ \mathbb{E}_{\ell}^\Phi\otimes f_2\circ \mathbb{E}_{\ell}^\Phi)(\Delta(x))\\
&=(f_1\otimes f_2\circ \mathbb{E}_{\ell}^\Phi)(\Delta(\mathbb{E}_{\ell}^\Phi(x)))\\
&=(f_1\otimes f_2\circ \mathbb{E}_{\ell}^\Phi)(\alpha(x)))=f_1\convnew f_2 (x)
\end{split}
\end{equation*}
which proves that $\mu_1\ast\mu_2|_{_{\mathscr{O}_{\mathbb{G}}(\Phi\backslash \mathbb{G}) }}=f_1\convnew f_2$. It now follows from Theorem \ref{Theorem: from Phi invariant functional to bi Phi invariant functional and vice versa} that $\mu_1\ast\mu_2$ must be the unique $\Phi$-bi-invariant extension of $f_1\convnew f_2$ to $\pag$. 
\end{proof}

\subsubsection{$\qg$-invariant operators on expected right coidalgebras}

%\label{Definition: G invariant maps}
Recall that  linear map $T:\mathscr{O}_{\mathbb{G}}(\Phi\backslash \mathbb{G}) \longrightarrow\mathscr{O}_{\mathbb{G}}(\Phi\backslash \mathbb{G}) $ is called $\qg$-invariant if $(T\otimes \mathrm{id})\circ\alpha=\alpha\circ T$, see Definition \ref{Definition: G invariant maps}.

\begin{remark}
This definition is motivated by the following observation:

If $G$ is a classical compact group, then all expected right coidalgebras of $(C(G),\Delta)$, where $\Delta$ is the canonical coproduct on $C(G)$, are of the form $C(K\backslash G)$, for some compact subgroup $K\subset G$.

A linear map $T:C(K\backslash G)\longrightarrow C(K\backslash G)$ is called $G$-invariant, if $T$ is covariant with respect to the canonical action of $G$ on $C(K\backslash G)$ \cite{liao04,liao15}, i.e. denoting the action of $G$ on $C(K\backslash G)$ by $G\ni g\mapsto \lambda_g\in Aut(C(K\backslash G))$, we have $T\circ \lambda_g=\lambda_g\circ T$ for all $g\in G$. 

Let $E_\pi(\cdot):=\int_G~\chi_{_{\pi}}(g)\lambda_g(\cdot) dg$, where $dg$ is the Haar measure of $G$, $\pi$ is an irreducible unitary representation of $G$ and $\chi_{_{\pi}}$ is its character.
It can be seen that $E_\pi$ is a completely bounded idempotent and $C(K\backslash G)_\pi:=\{E_\pi(f):~f\in C(K\backslash G)\}$ is the spectral subspace of $C(K\backslash G)$ for the action $\lambda$, corresponding to $\pi$.
Denoting $\mathscr{O}_G(K\backslash G):=\bigoplus_{\pi}C(K\backslash G)_\pi$ and $\lambda:C(K\backslash G)\longrightarrow C(K\backslash G)\otimes C(G)$ by $\lambda(f)(x,g):=\lambda_g(f)(x)$, it follows that $\lambda|_{_{\mathscr{O}_G(K\backslash G)}}:\mathscr{O}_G(K\backslash G)\longrightarrow \mathscr{O}_G(K\backslash G)\otimes \mathscr{O}_G(G)$ is a right coaction of the Hopf *-algebra $(\mathscr{O}_G(G),\Delta)$, where $\Delta$ is the restriction of the canonical coproduct on $C(G)$. 

Using the covariance of $T$, it is possible to see now that \mbox{$T(\mathscr{O}_G(K\backslash G))\subset \mathscr{O}_G(K\backslash G)$} and $(T\otimes \mathrm{id})\circ\lambda|_{_{\mathscr{O}_G(K\backslash G)}}=\lambda|_{_{\mathscr{O}_G(K\backslash G)}}\circ T$.
\end{remark}

\begin{lemma}\label{lemma: from invariant operators to Phi invariant functionals and vice versa}
Let $T:\mathscr{O}_{\mathbb{G}}(\Phi\backslash \mathbb{G}) \longrightarrow\mathscr{O}_{\mathbb{G}}(\Phi\backslash \mathbb{G}) $ be $\qg$-invariant. Then $\gamma:=\varepsilon\circ T$ is a $\Phi$-invariant functional on $\mathscr{O}_{\mathbb{G}}(\Phi\backslash \mathbb{G}) $.

Conversely, if $\gamma$ is a functional on $\mathscr{O}_{\mathbb{G}}(\Phi\backslash \mathbb{G}) $, then the formula $T:=(\gamma\otimes \mathbb{E}_{\ell}^\Phi)\circ\alpha$ defines a $\qg$-invariant map on $\mathscr{O}_{\mathbb{G}}(\Phi\backslash \mathbb{G}) $. However, $\gamma=\varepsilon\circ T$ if and only if $\gamma$ is $\Phi$-invariant.
\end{lemma}

\begin{proof}
The $\Phi$-invariance of $\gamma$ can be seen as follows:
\[
\begin{split}
(\gamma\otimes \Phi)(\alpha(x))&=(\varepsilon\circ T\otimes \Phi)(\alpha(x))\\
&=(\varepsilon\otimes \Phi)(\alpha(Tx))\\
&=\Phi(Tx)\\
&=\varepsilon(\mathbb{E}_{\ell}^\Phi(Tx))\\
&=\varepsilon(Tx)=\gamma(x)\quad(\text{as $Tx\in\mathscr{O}_{\mathbb{G}}(\Phi\backslash \mathbb{G}) $}).
\end{split}
\]

Now let $\gamma:\mathscr{O}_{\mathbb{G}}(\Phi\backslash \mathbb{G}) \longrightarrow\mathbb{C}$ be a functional. Then 

\begin{equation*}
\begin{split}
(T\otimes \mathrm{id})(\alpha(x))&=(\gamma\otimes \mathbb{E}_{\ell}^\Phi\otimes \mathrm{id})(\alpha\otimes \mathrm{id})(\alpha(x))\\
&=(\gamma\otimes \mathbb{E}_{\ell}^\Phi\otimes \mathrm{id})(\mathrm{id}\otimes\Delta)(\alpha(x))\\
&=\{\gamma\otimes(\mathbb{E}_{\ell}^\Phi\otimes \mathrm{id})\circ\Delta\}(\alpha(x))\\
&=(\gamma\otimes\Delta\circ \mathbb{E}_{\ell}^\Phi)(\alpha(x))\\
&=\Delta((\gamma\otimes \mathbb{E}_{\ell}^\Phi)(\alpha(x)))=\Delta(Tx),
\end{split}
\end{equation*}
which proves the $\qg$-invariance of $T$. 

We may now observe that 
\begin{equation*}
\begin{split}
\varepsilon(Tx)&=(\gamma\otimes \varepsilon\circ \mathbb{E}_{\ell}^\Phi)(\alpha)(x))=(\gamma\otimes h)(\alpha(x)), 
\end{split}
\end{equation*}

from which it follows that $\varepsilon\circ T=\gamma$ if and only if $\gamma$ is $\Phi$-invariant.
\end{proof}

The above lemma leads to the following observation:

\begin{theorem}\label{Theorem: one to one correspondence between Phi invariant functionals and invariant operators}
There exists a one-to-one correspondence between $\Phi$-invariant functionals on $\mathscr{O}_{\mathbb{G}}(\Phi\backslash \mathbb{G}) $ (denoted by $\gamma$) and $\qg$-invariant operators on $\mathscr{O}_{\mathbb{G}}(\Phi\backslash \mathbb{G}) $ (denoted by $T$), given by 
\[
\gamma\mapsto T:=(\gamma\otimes \mathbb{E}_{\ell}^\Phi)\circ\alpha; 
\]
\[
T\mapsto \gamma:=\varepsilon\circ T. 
\]
\end{theorem}

We now relate the convolution of $\Phi$-invariant functionals on $\mathscr{O}_{\mathbb{G}}(\Phi\backslash \mathbb{G}) $ with composition of $\qg$-invariant operators on $\mathscr{O}_{\mathbb{G}}(\Phi\backslash \mathbb{G}) $.

\begin{theorem}\label{Theorem: convolution of h-invariant functionals corresponds to composition of G invariant operators}

Let $\gamma_1$ and $\gamma_2$ be two $\Phi$-invariant functionals on $\mathscr{O}_{\mathbb{G}}(\Phi\backslash \mathbb{G}) $, and $T_1$ and $T_2$ be the corresponding $\qg$-invariant operators (or vice-versa as given by Theorem \ref{Theorem: one to one correspondence between Phi invariant functionals and invariant operators}). Then we have 
\[
\gamma_1\convnew\gamma_2=\varepsilon\circ T_2\circ T_1. 
\]
\end{theorem}

\begin{proof}
For $x\in\mathscr{O}_{\mathbb{G}}(\Phi\backslash \mathbb{G}) $ observe that 
\begin{equation*}
\begin{split}
(\gamma_1\convnew\gamma_2)(x)&=(\gamma_1\otimes\gamma_2\circ \mathbb{E}_{\ell}^\Phi)(\alpha(x))\\
&=(\varepsilon\circ T_1\otimes\varepsilon\circ T_2\circ E_1)(\alpha(x))\\
&=(\varepsilon\otimes\varepsilon\circ T_2\circ \mathbb{E}_{\ell}^\Phi)(\alpha(T_1x))\\
&=\varepsilon(T_2(\mathbb{E}_{\ell}^\Phi(T_1x)))\\
&=\varepsilon(T_2(T_1(x))),
\end{split}
\end{equation*}

which proves our claim.

\end{proof}

%%%%%%%%%%
\section{Markov semigroups on expected right coidalgebras} \label{sec-markovsg}
%%%%%%%%%%

As before, we fix $\Phi\in\mathcal{I}(\qg)$ and let $\mathbb{E}_{\ell}^\Phi:=(\Phi\otimes \mathrm{id})\circ\Delta$, $\mathbb{E}_r^\Phi:=( \mathrm{id}\otimes \Phi)\circ\Delta$, $\mathscr{O}_{\mathbb{G}}(\Phi\backslash \mathbb{G}) :=\mathbb{E}_{\ell}^\Phi(\pag)$ and $\mathscr{O}_{\mathbb{G}}(\mathbb{G}/ \Phi):=\mathbb{E}_r^\Phi(\pag)$.

A one parameter family of ($\qg$-invariant) operators  $T:=\{T_t:\mathscr{O}_{\mathbb{G}}(\Phi\backslash \mathbb{G}) \longrightarrow\mathscr{O}_{\mathbb{G}}(\Phi\backslash \mathbb{G})\}_{t\geq0}$ will be called a  semigroup of operators if $T_{t+s}=T_t\circ T_s$.

\subsection{Structure of convolution semigroups of invariant functionals on expected right coidalgebras}\label{Subsection: structure of convolution semigroup of invariant functionals on right coidalgebras}

The convolution on $\pquo$ allows us to define convolution semigroups of functionals or states on $\pquo$ in the same way as in Definition \ref{def-conv-sg-G}.

\begin{definition}\label{def-conv-sg-quo}
A \emph{convolution semigroup} on an expected rigth coidalgebra $\pqphi$ is a family of linear functionals $(\lambda_t:\pqphi\to\mathbb{C})_{t\ge 0}$ such that
\begin{enumerate}
\item
$\lim_{t\searrow 0} \lambda_t(a)=\lambda_0(a)$ for all $a\in\pquo$
(weak continuity);
\item
$\lambda_s \convnew \lambda_t = \lambda_{s+t}$ for all $s,t\ge 0$ (semigroup property).
\end{enumerate}
We call $(\lambda_t)_{t\ge 0}$ a \emph{convolution semigroup of states}, if the functionals $\lambda_t$ are furthermore normalized, i.e., $\lambda_t(1)=1$, and positive, i.e., $\lambda_t(a^*a)\ge 0$ for all $a\in\pqphi$ and all $t\ge0$.
\end{definition}

\begin{lemma}\label{Lemma: convolution semigroup of functionals on coidalgebra lifts to convolution semigroup of functionals on G}
Let $\lambda:=\{\lambda_t:\mathscr{O}_{\mathbb{G}}(\Phi\backslash \mathbb{G}) \longrightarrow\mathbb{C}\}_{t\geq0}$ be a convolution semigroup of $\Phi$-invariant functionals on $\mathscr{O}_{\mathbb{G}}(\Phi\backslash \mathbb{G}) $. For each $t\geq0$ let $\mu_t:\pag\longrightarrow\mathbb{C}$ be the unique $\Phi$-bi-invariant extension of $\lambda_t$, as given by Theorem \ref{Theorem: from Phi invariant functional to bi Phi invariant functional and vice versa}. Then $\mu:=\{\mu_t:\pag\longrightarrow\mathbb{C}\}_{t\geq0}$ is a convolution semigroup of functionals on $\pag$.
\end{lemma}

\begin{proof}
Fix $t,s\in[0,+\infty)$. It follows that \[(\mu_t\ast\mu_s)|_{_{\mathscr{O}_{\mathbb{G}}(\Phi\backslash \mathbb{G}) }}=\lambda_t\convnew\lambda_s=\lambda_{t+s}=\mu_{t+s}|_{_{\mathscr{O}_{\mathbb{G}}(\Phi\backslash \mathbb{G}) }},\] where the first equality follows from Theorem \ref{Theorem: Convolution at the group level restricts to give convolution at the coidalgebra level}-(b). By the same, we know that $\mu_t\ast\mu_s$ is the unique $\Phi$-bi-invariant extension of $\lambda_t\convnew\lambda_s=\lambda_{t+s}$. This implies that $\mu_t\ast\mu_s=\mu_{t+s}$. 

Weak continuity easily follows from the formula $\mu_t=\lambda_t\circ \mathbb{E}_\ell^\Phi$.
\end{proof}

\begin{remark}
In general the convolution semigroup $(\mu_t)_{t\ge0}$ does not start with the counit. Instead we have $\mu_0=\lambda_0\circ\mathbb{E}^\Phi_\ell = \Phi*\lambda_0$.
\end{remark}

We next prove an automatic $\Phi$-invariance of convolution semigroup of functionals, starting at a state.

\begin{lemma}\label{Lemma: convolution semigroup of functionals starting from a state on coidalgebra is Phi invariant}
Let $\{\lambda_t:\mathscr{O}_{\mathbb{G}}(\Phi\backslash \mathbb{G}) \longrightarrow\mathbb{C}\}_{t\geq0}$ be a convolution semigroup of functionals such that $\lambda_0$ is a state on $(\mathscr{O}_{\mathbb{G}}(\Phi\backslash \mathbb{G})$, i.e., $\lambda_0(x^*x)\ge 0$ and $\lambda_0(\mathbb{1})=1$. Then for each $t\geq0$, $\lambda_t$ is $\Phi$-invariant. 
\end{lemma}

\begin{proof}
Let $\mu_t:=\lambda_t\circ \mathbb{E}_{\ell}^\Phi$. Lemma \ref{Lemma: convolution semigroup of functionals on coidalgebra lifts to convolution semigroup of functionals on G} implies that $(\mu_t)_{t\geq0}$ is a convolution semigroup of functionals on $\pag$, such that for each $t\geq0$, $\mu_t$ is a left $\Phi$-invariant functional on $\pag$. Let us first show that $\mu_t$ is $\Phi$-bi-invariant.

We may note that $\mu_0:=\lambda_0\circ \mathbb{E}_{\ell}^\Phi$ is an idempotent state on $\pag$. Moreover, as $\mu_t$ is left $\Phi$-invariant for each $t\geq0$, this implies in particular that $\Phi\ast\mu_0=\mu_0$. Hence by Lemma \ref{Lemma: idempotent states commute}, we have $\mu_0\ast \Phi=\mu_0$. This implies that $\mu_t\ast(\mu_0\ast \Phi)=\mu_t$ i.e. $\mu_t\ast \Phi=\mu_t$ for all $t\geq0$. Hence $(\mu_t)_{t\geq0}$ is a convolution semigroup of $\Phi$-bi-invariant functionals on $\pag$. Theorem \ref{Theorem: from Phi invariant functional to bi Phi invariant functional and vice versa} now yields that $\lambda_t=\mu_t|_{_{\mathscr{O}_{\mathbb{G}}(\Phi\backslash \mathbb{G}) }}$ must be $\Phi$-invariant for each $t\geq0$. This proves the claim.
\end{proof}

We will now have a look at the differentiability properties of convolution semigroups on $\pqphi$ and the associated operator semigroups.

\begin{proposition}\label{Theorem: continuous convolution semigroups on coidalgebras are characterised by generators}
Let $(\lambda_t)_{t\geq0}$ be a pointwise continuous convolution semigroup of $\Phi$-invariant functionals on $\mathscr{O}_{\mathbb{G}}(\Phi\backslash \mathbb{G}) $. Then for each $x\in\mathscr{O}_{\mathbb{G}}(\Phi\backslash \mathbb{G}) $, the function $[0,+\infty)\ni t\mapsto\lambda_t(x)\in\mathbb{C}$ is differentiable at $t=0$.
\end{proposition}

\begin{proof}
Let $(\mu_t)_{t\ge0}$ be the unique $\Phi$-bi-invariant extension of $(\lambda_t)_{t\ge0}$. This is a continuous convolution semigroup of linear functionals and the discussion in \cite[Section 3]{franz+schurmann00} shows that it is differentiable, which implies the differentiability of $(\lambda_t)_{t\ge0}$.
\end{proof}

The following result is an `operator' version of Proposition \ref{Theorem: continuous convolution semigroups on coidalgebras are characterised by generators}.

\begin{proposition}\label{Theorem: strongly continous G invariant semigroups are differentiable}
Let $\{T_t:\mathscr{O}_{\mathbb{G}}(\Phi\backslash \mathbb{G}) \longrightarrow\mathscr{O}_{\mathbb{G}}(\Phi\backslash \mathbb{G}) \}_{t\geq0}$ be a poinwise continuous (w.r.t.\ to the universal C$^*$-norm) one parameter semigroup such that for each $t\geq0$, $T_t$ is $\qg$-invariant. Then for each $x\in\mathscr{O}_{\mathbb{G}}(\Phi\backslash \mathbb{G}) $, the map $[0,+\infty)\ni t\mapsto T_t(x)\in\mathscr{O}_{\mathbb{G}}(\Phi\backslash \mathbb{G}) $ is differentiable at $0$.
\end{proposition}

\begin{proof}
This follows by applying Proposition \ref{Theorem: continuous convolution semigroups on coidalgebras are characterised by generators} to $\lambda_t=\varepsilon\circ T_t$ and using  Theorem \ref{Theorem: one to one correspondence between Phi invariant functionals and invariant operators}.
\end{proof}

The next result is a converse of Proposition \ref{Theorem: continuous convolution semigroups on coidalgebras are characterised by generators}.

\begin{proposition}\label{Theorem: h-invariant functionals are generators of h-invariant convolution semigroup}
Let $\psi:\mathscr{O}_{\mathbb{G}}(\Phi\backslash \mathbb{G}) \longrightarrow\mathbb{C}$ be a $\Phi$-invariant map. Then there exists a strongly continuous convolution semigroup $\{\lambda_t:\mathscr{O}_{\mathbb{G}}(\Phi\backslash \mathbb{G}) \longrightarrow\mathbb{C}\}_{t\geq0}$ consisting of $\Phi$-invariant maps and $\lambda_0=\varepsilon|_{_{\mathscr{O}_{\mathbb{G}}(\Phi\backslash \mathbb{G}) }}$, such that $\psi=\left.\frac{d}{dt}\right|_{t=0}\lambda_t$.
\end{proposition}

\begin{proof}
For $x\in\mathscr{O}_{\mathbb{G}}(\Phi\backslash \mathbb{G}) $, define $A(x):=(\psi\otimes \mathbb{E}_{\ell}^\Phi)(\Delta(x))$. From Theorem \ref{Theorem: one to one correspondence between Phi invariant functionals and invariant operators}, it follows that $A:\mathscr{O}_{\mathbb{G}}(\Phi\backslash \mathbb{G}) \longrightarrow\mathscr{O}_{\mathbb{G}}(\Phi\backslash \mathbb{G}) $ is a $\qg$-invariant operator. 

Fix $x\in\mathscr{O}_{\mathbb{G}}(\Phi\backslash \mathbb{G})$. We can use fundamental theorem of coalgebras to restrict to finite-dimensional subcoalgebra $X$ that contains $x$, one sees that  
\[
T_t(x):=\sum_{k=0}^\infty \frac{t^k}{k!}A^k(x)
\]
converges on $X$. Since $x$ was arbitrary, the convergence holds for all $x\in\pqphi$ and defines a semigroup of $\qg$-invariant operators.

Let $\lambda_t:=\varepsilon\circ T_t$ for each $t$. An application of Theorem \ref{Theorem: one to one correspondence between Phi invariant functionals and invariant operators} and Lemma \ref{Lemma: convolution semigroup of functionals on coidalgebra lifts to convolution semigroup of functionals on G} implies that $(\lambda_t)_{t\geq0}$ is a convolution semigroup of $\Phi$-invariant functionals on $\mathscr{O}_{\mathbb{G}}(\Phi\backslash \mathbb{G}) $. Since $X$ is finite dimensional, $\varepsilon|_{_{X}}$ is a bounded functional on $X$. From this, it follows easily that the map $[0,+\infty)\ni t\mapsto \lambda_t(x)$ is continuous at $0$. The result follows now.
\end{proof}
\begin{corollary}\label{Corollary: Phi bi invariant functionals on G are generators of Phi bi invariant convolution semigroups of functionals on G}
Let $\Phi\in\mathcal{I}(\qg)$ and $\psi:\pag\longrightarrow\mathbb{C}$ be a functional which is $\Phi$-bi-invariant. Then there exists a convolution semigroup of functionals $\{\lambda_t:\pag\longrightarrow\mathbb{C}\}_{t\geq0}$ such that for each $t\geq0$, $\lambda_t$ is $\Phi$-bi-invariant, and $\lambda_0=\Phi$ and $\frac{d}{dt}|_{_{t=0}}\lambda_t=\psi$.
\end{corollary}

\begin{proof}
Let $\mathbb{E}_{\ell}^\Phi:=(\Phi\otimes \mathrm{id})\circ\Delta$ and $\mathscr{O}_{\mathbb{G}}(\Phi\backslash \mathbb{G}) :=\mathbb{E}_{\ell}^\Phi(\pag)$. Then $\phi:=\psi|_{_{\mathscr{O}_{\mathbb{G}}(\Phi\backslash \mathbb{G}) }}$ is a $\Phi$-invariant functional on the expected right coidalgebra $\mathscr{O}_{\mathbb{G}}(\Phi\backslash \mathbb{G}) $. Then by Theorem \ref{Theorem: h-invariant functionals are generators of h-invariant convolution semigroup}, it follows that there exists a convolution semigroup $\{\beta_t:\mathscr{O}_{\mathbb{G}}(\Phi\backslash \mathbb{G}) \longrightarrow\mathbb{C}\}_{t\geq0}$ such that for each $t\geq0$, $\beta_t$ is a $\Phi$-bi-invariant functional on $\mathscr{O}_{\mathbb{G}}(\Phi\backslash \mathbb{G})$, $\beta_0=\varepsilon|_{_{\mathscr{O}_{\mathbb{G}}(\Phi\backslash \mathbb{G}) }}$ and $\frac{d}{dt}|_{_{t=0}}\beta_t=\phi$. Let $\{\lambda_t:\pag\longrightarrow\mathbb{C}\}_{t\geq0}$ be the extension of $(\beta_t)_{t\geq0}$ to a $\Phi$-bi-invariant convolution semigroup of functionals on $\pag$, as given by Lemma \ref{Lemma: convolution semigroup of functionals on coidalgebra lifts to convolution semigroup of functionals on G}. It now follows that $(\lambda_t)_{t\geq0}$ is the required convolution semigroup with the desired property.
\end{proof} 

\subsection{Structure of convolution semigroups of states on expected coidalgebras}

\begin{remark}\label{Remark: one to one correspondence between convolution semigroup of invariant functionals and invariant Markov semigroups}
It is worthwhile to note at this point that Theorem \ref{Theorem: one to one correspondence between Phi invariant functionals and invariant operators} along with Lemma \ref{Lemma: convolution semigroups of functionals starting from a state on coidalgebra is Phi invariant} essentially gives us a way to go back and forth between convolution semigroup of states on $\mathscr{O}_{\mathbb{G}}(\Phi\backslash \mathbb{G}) $ and $\qg$-invariant Markov semigroup on $\mathscr{O}_{\mathbb{G}}(\Phi\backslash \mathbb{G}) $.
\end{remark}
The following theorem gives a Schoenberg correspondence for expected right coidalgebras. 

\begin{theorem}\label{Theorem: Schoenberg correspondence for an expected right coidalgebra}
Let $\{\lambda_t:\mathscr{O}_{\mathbb{G}}(\Phi\backslash \mathbb{G}) \longrightarrow\mathbb{C}\}_{t\geq0}$ be a strongly continous convolution semigroup of functionals. Let $\psi:=\frac{d}{dt}\lambda_t|_{_{t=0}}$. Then the following are equivalent.
\begin{itemize}
\item[(i)]
$(\lambda_t)_{t\geq0}$ is a convolution semigroup of states.
\item[(ii)]
$\psi$ is a well-defined map on $\mathscr{O}_{\mathbb{G}}(\Phi\backslash \mathbb{G})$, $\lambda_0$ is positive and $\psi(x^*x)\geq0$ for all $x\in\mathscr{O}_{\mathbb{G}}(\Phi\backslash \mathbb{G})$ with $\lambda_0(x^*x)=0$, and $\psi(x^\ast)=\overline{\psi(x)}$ for all $x\in\mathscr{O}_{\mathbb{G}}(\Phi\backslash \mathbb{G})$.
\end{itemize}
\end{theorem}

\begin{proof}
Let us first extend $(\lambda_t)_{t\geq0}$ to a convolution semigroup \mbox{$\{\mu_t:\pag\longrightarrow\mathbb{C}\}_{t\geq0}$} of $\Phi$-bi-invariant functionals, as shown in Lemma \ref{Lemma: convolution semigroup of functionals starting from a state on coidalgebra is Phi invariant}. As in the proof of Theorem \ref{Theorem: continuous convolution semigroups on coidalgebras are characterised by generators}, $(\mu_t)_{t\geq0}$ is also strongly continuous. Moreover, since for each $t\in[0,+\infty)$, $\mu_t=\lambda_t\circ \mathbb{E}_{\ell}^\Phi$, this implies that $\mu_t$ is a state on $\pag$ for all $t$. Moreover, from the proof of Theorem \ref{Theorem: continuous convolution semigroups on coidalgebras are characterised by generators} it follows that $\psi\circ \mathbb{E}_{\ell}^\Phi=\frac{d}{dt}\mu_t|_{t=0}$. So it is enough to prove (i) and (ii) for $(\mu_t)_{t\geq0}$. Since $(\pag,\Delta)$ is a *-bialgebra, the result now follows from \cite[Theorem 3.3]{franz+schurmann00}.
\end{proof}

%%%%%%%%%%
\section{Quantum hypergroups} \label{sec-qhyper}
%%%%%%%%%%

\subsection{Functionals on the algebra of $\Phi$-bi-invariant functions on $\qg$}\label{Subsection: Functionals on the algebra of h-bi-invariant functions on G}

Let $ \Phi \in\mathcal{I}(\qg)$ and denote $\mathbb{E}_{\ell}^\Phi:=( \Phi \otimes{\rm id})\circ\Delta$ and $\mathbb{E}_r^\Phi:=({\rm id}\otimes  \Phi)\circ\Delta$. Let $\mathscr{O}_{\mathbb{G}}(\Phi\backslash \mathbb{G}) :=\mathbb{E}_{\ell}^\Phi(\pag)$ and $\mathscr{O}_{\mathbb{G}}(\mathbb{G}/ \Phi):=\mathbb{E}_r^\Phi(\pag)$.

\begin{definition}\label{Definition: definition of h-bi-invariant functions on qg}
The \emph{*-algebra of $\Phi$-bi-invariant functions} on $\qg$, denoted by $\mathscr{O}_{\mathbb{G}}(\Phi \backslash\mathbb{G}/ \Phi)$ is defined by $\mathscr{O}_{\mathbb{G}}(\Phi \backslash\mathbb{G}/ \Phi):=\mathscr{O}_{\mathbb{G}}(\Phi\backslash \mathbb{G}) \cap\mathscr{O}_{\mathbb{G}}(\mathbb{G}/ \Phi)=\mathbb{E}_{\ell}^\Phi(\mathbb{E}_r^\Phi(\pag))$.
\end{definition}

\begin{remark}
It is worthwhile to note that in the context of CQG algebras, the double coset hyper bialgebra considered in \cite{franz+schurmann00} is a special case of the algebra introduced in Definition \ref{Definition: definition of h-bi-invariant functions on qg}. Haonan Zhang \cite[Proposition 2.4]{haonan} has shown that $C(\Phi\backslash \qg/\Phi)=\mathbb{E}_{\ell}^\Phi(\mathbb{E}_r^\Phi(C(\qg)))$ has the structure of a compact quantum hypergroup in the sense of \cite{chap+va99}.
\end{remark}

\begin{theorem}\label{Theorem: from functionals on dch bi algebra to Phi bi invariant functionals on qg}
Let $f$ be a functional on $\mathscr{O}_{\mathbb{G}}(\Phi \backslash\mathbb{G}/ \Phi)$ and define $\mu:=f\circ \mathbb{E}_{\ell}^\Phi\circ \mathbb{E}_r^\Phi$. Then $\mu$ is the unique $\Phi$-bi-invariant functional on $\pag$ such that $\mu|_{_{\mathscr{O}_{\mathbb{G}}(\Phi \backslash\mathbb{G}/ \Phi)}}=f$.
\end{theorem}

\begin{proof}
We prove the $\Phi$-bi-invariance of $\mu$ as a functional on $\pag$. We may note that an easy computation yields $\mathbb{E}_r^\Phi\circ \mathbb{E}_{\ell}^\Phi=\mathbb{E}_{\ell}^\Phi\circ \mathbb{E}_r^\Phi$. We only show the left $\Phi$-invariance of $\mu$. The proof of right $\Phi$-invariance is identical. Let $x\in\pag$.
\begin{equation*}
\begin{split}
( \Phi \otimes\mu)(\Delta(x))&=( \Phi \otimes f\circ \mathbb{E}_{\ell}^\Phi\circ \mathbb{E}_r^\Phi)(\Delta(x)) \\
&=f(\mathbb{E}_{\ell}^\Phi(\mathbb{E}_r^\Phi(\mathbb{E}_{\ell}^\Phi(x))))\\
&=f(\mathbb{E}_r^\Phi((\mathbb{E}_{\ell}^\Phi)^2(x)))\\
&=f(\mathbb{E}_r^\Phi(\mathbb{E}_{\ell}^\Phi(x)))\\
&=f(\mathbb{E}_{\ell}^\Phi(\mathbb{E}_r^\Phi(x)))=\mu(x),
\end{split}
\end{equation*}
which proves left $\Phi$-invariance of $\mu$.

Conversely, suppose $\nu$ is a $\Phi$-invariant functional on $\pag$ such that $\nu|_{_{\mathscr{O}_{\mathbb{G}}(\Phi \backslash\mathbb{G}/ \Phi)}}=f$. Let $x\in\pag$. We have
\begin{equation*}
\begin{split}
\nu(x)&=( \Phi \otimes\nu)(\Delta(x)) \\
&=\nu(\mathbb{E}_{\ell}^\Phi(x)) \\
&=(\nu\otimes  \Phi )(\Delta(\mathbb{E}_{\ell}^\Phi(x))) \\
&=\nu(\mathbb{E}_r^\Phi(\mathbb{E}_{\ell}^\Phi(x))) \\
&=f(\mathbb{E}_{\ell}^\Phi(\mathbb{E}_r^\Phi(x)))=\mu(x),
\end{split}
\end{equation*}
which proves the uniqueness.
\end{proof}

A functional on $f$ on $\mathscr{O}_{\mathbb{G}}(\Phi \backslash\mathbb{G}/ \Phi)$ can be extended in many ways to a functional on the right coidalgebra $\mathscr{O}_{\mathbb{G}}(\Phi\backslash \mathbb{G}) $. For example, let $x\in\mathscr{O}_{\mathbb{G}}(\Phi\backslash \mathbb{G}) $. Then $x$ admits a unique decomposition $x=a+b$, where $a\in \mathbb{E}_r^\Phi(\pag)$ and \mbox{$b\in (\mathbb{E}_r^\Phi)^\perp(\pag)$}. Note that $a\in\mathscr{O}_{\mathbb{G}}(\Phi \backslash\mathbb{G}/ \Phi)$. Now the assignment \mbox{$x\mapsto f(a)+\psi(b)$}, for any functional $\psi$ on $(\mathbb{E}_r^\Phi)^\perp(\pag)$, gives a well-defined functional on $\mathscr{O}_{\mathbb{G}}(\Phi\backslash \mathbb{G}) $. However, not all such extensions will be $\Phi$-invariant as functionals on $\mathscr{O}_{\mathbb{G}}(\Phi\backslash \mathbb{G}) $. In fact we have

\begin{corollary}\label{Corollary: from functionals on dch bialgebra to Phi invariant functionals on coidalgebra}
Let $f$ be a functional on $\mathscr{O}_{\mathbb{G}}(\Phi \backslash\mathbb{G}/ \Phi)$. Then there exists a unique functional $\lambda$ on $\mathscr{O}_{\mathbb{G}}(\Phi\backslash \mathbb{G}) $ such that
\begin{itemize}
\item
$\lambda|_{_{\mathscr{O}_{\mathbb{G}}(\Phi \backslash\mathbb{G}/ \Phi)}}=f$.
\item $\lambda$ is a $\Phi$-invariant functional on $\mathscr{O}_{\mathbb{G}}(\Phi\backslash \mathbb{G}) $ in the sense of Definition \ref{Definition: h invariance and h-bi-invarince}.
\end{itemize}
\end{corollary}

\begin{proof}
Let us first prove that there exists at least one $\Phi$-invariant extension of $f$. By virtue of Theorem \ref{Theorem: from functionals on dch bi algebra to Phi bi invariant functionals on qg}, we see $\mu:=f\circ \mathbb{E}_{\ell}^\Phi\circ \mathbb{E}_r^\Phi$ is a $\Phi$-bi-invariant functional on $\pag$. Thus by Theorem \ref{Theorem: from Phi invariant functional to bi Phi invariant functional and vice versa} we see that $\lambda:=\mu|_{_{\mathscr{O}_{\mathbb{G}}(\Phi\backslash \mathbb{G}) }}$ is a $\Phi$-invariant functional on $\mathscr{O}_{\mathbb{G}}(\Phi\backslash \mathbb{G}) $. Clearly $\lambda|_{_{\mathscr{O}_{\mathbb{G}}(\Phi \backslash\mathbb{G}/ \Phi)}}=f$, which proves that there exists at least one $\Phi$-invariant extension of $f$. 

Suppose $\lambda^\prime:\mathscr{O}_{\mathbb{G}}(\Phi\backslash \mathbb{G}) \longrightarrow\mathbb{C}$ be another $\Phi$-invariant extension of $f$. Let us suppose that $\mu_1:\pag\longrightarrow\mathbb{C}$ be the unique $\Phi$-bi-invariant extension of $f$ given by Theorem \ref{Theorem: from functionals on dch bi algebra to Phi bi invariant functionals on qg}, and $\mu_2:\pag\longrightarrow\mathbb{C}$ be the unique $\Phi$-bi-invariant extension of $\lambda^\prime$ as given by Theorem \ref{Theorem: from Phi invariant functional to bi Phi invariant functional and vice versa}. Since $\mu_2|_{_{\mathscr{O}_{\mathbb{G}}(\Phi \backslash\mathbb{G}/ \Phi)}}=\lambda^\prime|_{_{\mathscr{O}_{\mathbb{G}}(\Phi \backslash\mathbb{G}/ \Phi)}}=f$, this implies that $\mu_2$ is also a $\Phi$-bi-invariant extension of $f$. By the uniqueness of such an extension as shown in Theorem \ref{Theorem: from functionals on dch bi algebra to Phi bi invariant functionals on qg}, we must have $\mu_1=\mu_2$ which in turn implies that $\lambda^\prime=\lambda$.
\end{proof}

Theorem \ref{Theorem: from functionals on dch bi algebra to Phi bi invariant functionals on qg} and Corollary \ref{Corollary: from functionals on dch bialgebra to Phi invariant functionals on coidalgebra} together yield:

All functionals on the *-algebra of $\Phi$-bi-invariant functions on $\qg$ are precisely the restrictions of $\Phi$-bi-invariant functionals on $\pag$. Hence they are also restrictions of $\Phi$-invariant functionals on the corresponding right coidalgebra.

\subsection{Convolution of functionals on the *-algebra of $\Phi$-bi-invariant functions on $\qg$}

In this subsection, we again consider the *-algebra of $\Phi$-bi-invariant functions on $\qg$ denoted by $\mathscr{O}_{\mathbb{G}}(\Phi \backslash\mathbb{G}/ \Phi)$, as defined in Definition \ref{Definition: definition of h-bi-invariant functions on qg}. We will define a coproduct on $\mathscr{O}_{\mathbb{G}}(\Phi \backslash\mathbb{G}/ \Phi)$, which will turn it into a *-bi-algebra.
\begin{definition}\label{Definition: of coproduct for the algebra if Phi bi invariant functions on qg}
Define $\widetilde{\Delta}:\pag\longrightarrow\pag\otimes\pag$ by 
\[
\widetilde{\Delta}(x):=({\rm id}\otimes  \Phi \otimes \mathrm{id})(\Delta^{(2)}(x)),\qquad x\in\mathscr{O}_{\mathbb{G}}(\Phi \backslash\mathbb{G}/ \Phi), 
\]
where $\Delta^{(2)}=({\rm id}\otimes\Delta)\circ\Delta$.
\end{definition}

\begin{lemma}\label{Lemma: coproduct for the algebra of bi invariant functions on qg}
The triple $(\mathscr{O}_{\mathbb{G}}(\Phi \backslash\mathbb{G}/ \Phi),\widetilde{\Delta}|_{_{\mathscr{O}_{\mathbb{G}}(\Phi \backslash\mathbb{G}/ \Phi)}},\varepsilon|_{_{\mathscr{O}_{\mathbb{G}}(\Phi \backslash\mathbb{G}/ \Phi)}})$ is a \emph{hyper-bialgebra} (in the sense of \cite{franz+schurmann00}), i.e.,
\begin{enumerate}
\item
$\mathscr{O}_{\mathbb{G}}(\Phi \backslash\mathbb{G}/ \Phi)$ is a unital *-algebra;
\item
the triple $(\mathscr{O}_{\mathbb{G}}(\Phi \backslash\mathbb{G}/ \Phi),\widetilde{\Delta}|_{_{\mathscr{O}_{\mathbb{G}}(\Phi \backslash\mathbb{G}/ \Phi)}},\varepsilon|_{_{\mathscr{O}_{\mathbb{G}}(\Phi \backslash\mathbb{G}/ \Phi)}})$ is a coalgebra;
\item
the comultiplication $\widetilde{\Delta}|_{_{\mathscr{O}_{\mathbb{G}}(\Phi \backslash\mathbb{G}/ \Phi)}}$ is completely positive and the counit $\varepsilon|_{_{\mathscr{O}_{\mathbb{G}}(\Phi \backslash\mathbb{G}/ \Phi)}}$ is a *-algebra homomorphism.
\end{enumerate}
\end{lemma}
\begin{proof}
It follows easily that $\widetilde{\Delta}$ is completely positive and coassociative. We only need to show that $\widetilde{\Delta}(\mathscr{O}_{\mathbb{G}}(\Phi \backslash\mathbb{G}/ \Phi))\subset\mathscr{O}_{\mathbb{G}}(\Phi \backslash\mathbb{G}/ \Phi)\otimes\mathscr{O}_{\mathbb{G}}(\Phi \backslash\mathbb{G}/ \Phi)$. So let $x\in\mathscr{O}_{\mathbb{G}}(\Phi \backslash\mathbb{G}/ \Phi)$. We have 
\begin{equation*}
\begin{split}
\widetilde{\Delta}(x)&=(\mathbb{E}_r^\Phi\otimes \mathrm{id})(\Delta^{(2)}(x))\\
&=(\mathbb{E}_r^\Phi\otimes \mathrm{id})(\Delta(\mathbb{E}_r^\Phi\circ\mathbb{E}_{\ell}^\Phi(x)))\\
&=(\mathbb{E}_r^\Phi\circ\mathbb{E}_{\ell}^\Phi\otimes \mathbb{E}_r^\Phi)(\Delta(x))\quad(\text{by (a) of Lemma \ref{Lemma: useful 1st observation}})\\
&=(\mathbb{E}_{\ell}^\Phi\circ\mathbb{E}_r^\Phi\circ\mathbb{E}_r^\Phi\otimes \mathbb{E}_r^\Phi)(\Delta(x))\quad(\text{using $\mathbb{E}_r^\Phi\mathbb{E}_{\ell}^\Phi=\mathbb{E}_{\ell}^\Phi\mathbb{E}_r^\Phi$ and $(\mathbb{E}_r^\Phi)^2=\mathbb{E}_r^\Phi$})\\
&=(\mathbb{E}_{\ell}^\Phi\circ\mathbb{E}_r^\Phi\otimes \mathbb{E}_r^\Phi\circ\mathbb{E}_{\ell}^\Phi)(\Delta(x))\quad(\text{by (b) of Lemma \ref{Lemma: useful 1st observation}}).
\end{split}
\end{equation*}
{}From the last expression one can conclude that $\widetilde{\Delta}(x)\in\mathscr{O}_{\mathbb{G}}(\Phi \backslash\mathbb{G}/ \Phi)\otimes\mathscr{O}_{\mathbb{G}}(\Phi \backslash\mathbb{G}/ \Phi)$.
\end{proof}
As a consequence we can define convolution of functionals on $\mathscr{O}_{\mathbb{G}}(\Phi \backslash\mathbb{G}/ \Phi)$.
\begin{definition}\label{Definition: definition of convolution in the algebra of biinvariant functions on qg}
Let $f,g$ be two functionals on $\mathscr{O}_{\mathbb{G}}(\Phi \backslash\mathbb{G}/ \Phi)$. We define the convolution of $f$ and $g$ as the following functional:
\[
f\conva g:=(f\otimes g)\circ\widetilde{\Delta}.
\]
\end{definition}

Alternatively, we have $f\conva g:=f\star\Phi\star g$.

\begin{theorem}\label{Theorem: Convolution at the coidalgebra level restricts to convolution on cla and also G invariant operators}

The following holds:
\begin{itemize}
\item[(a)]
Let $f,g$ be functionals on $\mathscr{O}_{\mathbb{G}}(\Phi \backslash\mathbb{G}/ \Phi)$ and $\lambda_1,\lambda_2$ be their unique $\Phi$-invariant extensions as functionals on $\mathscr{O}_{\mathbb{G}}(\Phi\backslash \mathbb{G}) $ (given by Corollary \ref{Corollary: from functionals on dch bialgebra to Phi invariant functionals on coidalgebra}). Then 
\[
\lambda_1\convnew\lambda_2|_{_{\mathscr{O}_{\mathbb{G}}(\Phi \backslash\mathbb{G}/ \Phi)}}=f\conva g. 
\]
\item[(b)]
Let $T:\mathscr{O}_{\mathbb{G}}(\Phi \backslash\mathbb{G}/ \Phi)\longrightarrow\mathscr{O}_{\mathbb{G}}(\Phi \backslash\mathbb{G}/ \Phi)$ be a linear map such that $(T\otimes \mathrm{id})\circ\widetilde{\Delta}=\widetilde{\Delta}\circ T$. Then there exists a $\qg$-invariant map $S:\mathscr{O}_{\mathbb{G}}(\Phi\backslash \mathbb{G}) \longrightarrow\mathscr{O}_{\mathbb{G}}(\Phi\backslash \mathbb{G}) $ such that $S|_{_{\mathscr{O}_{\mathbb{G}}(\Phi \backslash\mathbb{G}/ \Phi)}}=T$. 
\end{itemize}
\end{theorem}

\begin{proof}
Observe that for $x\in\mathscr{O}_{\mathbb{G}}(\Phi \backslash\mathbb{G}/ \Phi)$,
\begin{equation*}
\begin{split}
(\lambda_1\convnew\lambda_2)(x)&=(\lambda_1\otimes\lambda_2\circ \mathbb{E}_{\ell}^\Phi)(\alpha(x)) \\
&=(\lambda_1\otimes\lambda_2)( \mathrm{id}\otimes \mathbb{E}_{\ell}^\Phi)(\Delta(x)) \\
&=(\lambda_1\otimes\lambda_2)(\mathbb{E}_r^\Phi\otimes \mathrm{id})(\Delta(x))\quad(\text{by (b) of Lemma \ref{Lemma: useful 1st observation}})\\
&=(\lambda_1\otimes\lambda_2)(\widetilde{\Delta}(x))\\
&=(f\otimes g)(\widetilde{\Delta}(x))=(f\conva g)(x),
\end{split}
\end{equation*}
which proves (a).

To prove (b), observe that the identity $(T\otimes \mathrm{id})\circ\widetilde{\Delta}=\widetilde{\Delta}\circ T$ implies that the functional $f:=\varepsilon\circ T$ satisfies $(f\otimes \mathrm{id})\circ\widetilde{\Delta}=T$.
Sine $f$ is a functional on $\mathscr{O}_{\mathbb{G}}(\Phi \backslash\mathbb{G}/ \Phi)$, by virtue of Corollary \ref{Corollary: from functionals on dch bialgebra to Phi invariant functionals on coidalgebra}, it extends to a $\Phi$-invariant functional $\lambda$ on $\mathscr{O}_{\mathbb{G}}(\Phi\backslash \mathbb{G})$.
Let $S:=(\lambda\otimes \mathbb{E}_{\ell}^\Phi)\circ\alpha$, which is a $\qg$-invariant operator on $\mathscr{O}_{\mathbb{G}}(\Phi\backslash \mathbb{G}) $, by virtue of Theorem \ref{Theorem: one to one correspondence between Phi invariant functionals and invariant operators}. Now an easy computation yields that $S|_{_{\mathscr{O}_{\mathbb{G}}(\Phi \backslash\mathbb{G}/ \Phi)}}=T$.
\end{proof}

%%%%%%%%%%
\section{Summary of the one-to-one correspondences}\label{sec-summary}
%%%%%%%%%%

We have established the following one-to-one correspondences.

\begin{theorem}\label{thm-corr-alg}
Let $\qg$ be a compact quantum group, $\Phi\in\mathcal{I}$ an idempotent state on $\qg$, and denote by $\qx=\Phi\backslash\qg$ the associated quantum space. Let $\delta=\varepsilon|_{\pqphi}$.

Then we have one-to-one correspondences between the following objects.
\begin{enumerate}
\item
semigroups of $\qg$-invariant operators on $\pqphi$ such that $(\delta\circ T_t)_{t\ge0}$ is weakly continuous;
\item
$\qg$-invariant convolution semigroups of linear functionals on $\pqphi$;
\item
$\Phi$-bi-invariant convolution semigroups of linear functionals on $\pag$;
\item
convolution semigroups of linear functionals on $\mathscr{O}_\qg(\Phi\backslash \qg/\Phi)$.
\end{enumerate}
\end{theorem}

If we add positivity, we can formulate the following one-to-one correspondences.

\begin{theorem}\label{thm-corr-pos}
Let $\qg$ be a compact quantum group, $\Phi\in\mathcal{I}$ an idempotent state on $\qg$. We have one-to-one correspondences between the following objects.
\begin{enumerate}
\item
$\qg$-invariant quantum Markov semigroups on $\pqphi$;
\item
$\qg$-invariant convolution semigroups of states on $\pqphi$;
\item
$\Phi$-bi-invariant convolution semigroups of states on $\pag$;
\item
convolution semigroups of states on $\mathscr{O}_\qg(\Phi\backslash \qg/\Phi)$.
\end{enumerate}
\end{theorem}

All these semigroups are furthermore characterized by their derivatives at $t=0$.

\begin{definition}\label{def-gen-funct}
Let $\mathcal{A}$ be a unital *-algebra and $\phi:\mathcal{A}\to\mathbb{C}$ a state. A linear functional $\psi:\mathcal{A}\to\mathbb{C}$ is called a \emph{$\phi$-generating functional}, if
\begin{enumerate}
\item
$\psi$ is normalised, i.e., $\psi(1)=0$;
\item
$\psi$ is hermitian, i.e., $\psi(a^*)=\overline{\psi(a)}$, for all $a\in\mathcal{A}$;
\item
$\psi$ is $\phi$-conditionally positive, i.e., $\psi(a^*a)\ge 0$ for all $a\in\mathcal{A}$ with $\phi(a^*a)=0$.
\end{enumerate}
\end{definition}

\begin{theorem}\label{thm-corr-gen-funct}
Under the same hypotheses as Theorem \ref{thm-corr-pos}, the objects in Theorem \ref{thm-corr-pos} are furthermore in one-to-one correspondence with
\begin{enumerate}
\item
$\qg$-invariant quantum Markov semigroups on $\pqphi$;
\item
$\qg$-invariant $\Phi|_{\pqphi}$-generating functionals on $\pqphi$;
\item
$\Phi$-bi-invariant on $\Phi$-generating functionals on $\pag$;
\item
$\varepsilon|_{\mathscr{O}_\qg (\Phi\backslash \qg/\Phi)}$-generating functionals on $\mathscr{O}_\qg(\Phi\backslash \qg/\Phi)$.
\end{enumerate}
\end{theorem}

In the examples in the next section we will determine the $\varepsilon|_{\mathscr{O}_\qg (\Phi\backslash \qg/\Phi)}$-generating functionals on $\mathscr{O}_\qg(\Phi\backslash \qg/\Phi)$ for the case where $\qg$ is one of the orthogonal quantum groups $O_N$, $O^*_N$ or $O^+_N$ and $\Phi$ is the idempotent state such that $\qx=\Phi\backslash \qg$ is one of the quantum spheres $S^{N-1}$, $S^{N-1}_*$, or $S^{N-1}_+$.

%%%%%%%%%%
\section{Markov semigroups on quantum spheres}\label{sec-spheres}
%%%%%%%%%%

We know that orthogonal group $O_N$ is the isometry group of  sphere $S^{N-1}$. There exist quantum versions, or ``liberated'' versions, of the orthogonal group and the sphere. These are given by their universal C$^*$-algebras which are defined as follows \cite{banica16}:
$$C^u(S^{N-1}_+)=C^*\left(x_1,\cdots,x_N | x_i=x_i^*, \sum_i x_i^2=1\right)$$
$$C^u(O_N^+)=C^*\left((u_{ij})_{i,j=1,\dots,N} |u=\bar{u}, u^t=u^{-1}\right)$$
$$C^u(S^{N-1}_*) =C^u(S^{N-1}_+)/<abc=cba, \forall a,b,c\in {x_i}>$$
$$C^u(O_N^*)=C^n(O_N^+)/<abc=cba, \forall a,b,c\in {u_{ij}}>.$$
We will use the notation $S^{N-1}_\times$ denote the three spheres above associated to the three quantum isometry groups $O_N^\times$, $\times\in\{\emptyset,*,+\}$ (where $\emptyset$ stands for no symbol). There exist unique actions $\alpha^u_\times:C^u(S^{N-1}_\times)\to C^u(S^{N-1}_\times)\otimes C^u(O^\times_N)$ of the three families orthogonal quantum groups on the corresponding spheres, such that $\alpha^u(x_i)=\sum_{j=1}^N x_j\otimes u_{ji}$ for $i\in\{1,\ldots,N\}$. Every such ``universal'' action comes with a ``reduced'' action $\alpha^r_\times:C(S^{N-1}_\times)\to C(S^{N-1}_\times)\otimes C(O^\times_N)$ and a Hopf-*-algebraic action $\alpha^H_\times:\mathscr{O}_{\mathbb{G}}(S^{N-1}_\times)\to \mathscr{O}_{\mathbb{G}}(S^{N-1}_\times)\otimes \mathscr{O}(O^\times_N)$, cf.\ \cite{decommer16}. 

Banica \cite[Proposition 5.8]{banica16} showed that the reduced function algebras of these spheres can regarded as subalgebras of the reduced function algebras of orthogonal groups. I.e. if we identify $x_i=u_{1i}$, then $C(S^{N-1}_\times) \subset C(O_N^{\times})$. One can check that $C(S^{N-1}_\times)$ is a coidalgebra of $C(O_N^{\times})$, so  we can define the corresponding idempotent state $\Phi$ such that the associated left, right, and two-sided conditional expectations $\mathbb{E}_{\ell}^\Phi$ , $\mathbb{E}_r^\Phi$, $\mathbb{E}_{bi}^\Phi$ satisfy:
\begin{gather*}
\mathscr{O}(S_\times^{N-1}) =\mathscr{O}_{O_N^\times}(\Phi\backslash O_N^\times)=\mathbb{E}_{\ell}^\Phi (\mathscr{O}(O_N^\times)) = \mbox{*-alg}\{u_{11},\ldots,u_{1N}\}, \\
S(\mathscr{O}(S_\times^{N-1})) =\mathscr{O}_{O_N^\times}(O_N^\times/\Phi) = \mathbb{E}_r^\Phi  (\mathscr{O}(O_N^\times)) = \mbox{*-alg}\{u_{11},\ldots,u_{N1}\}, \\
\mathscr{O}(S_\times^{N-1}) \cap S(\mathscr{O}(S_\times^{N-1}))=\mathscr{O}_{O_N^\times}(\Phi \backslash O_N^\times/\Phi) = \mathbb{E}_{bi}^\Phi  \big(\mathscr{O}(O_N^\times)\big) = \mbox{*-alg}\{u_{11}),
\end{gather*}
(where $S$ denotes the antipode of $\mathscr{O}(O^\times_N)$).

We know that in the classical case $S^{N-1}\cong O_{N-1}\backslash O_N$. Banica, Skalski, and So{\l}tan \cite{banica+skalski+soltan12} have shown that $S_+^{N-1}$ is not equal to the quotient $O^+_{N-1}\backslash O^+_N$. We will now show that the half-liberated and the free spheres can not be obtained as quotient spaces.
\begin{proposition}\label{prop-not-quotient}
There exists no quantum subgroup $\mathbb{K}$ of $O_N^+$ (or $O_N^*$, resp.) such that $S_+^{N-1}\cong O^+_N/\mathbb{K}$ (or $S_*^{N-1}\cong O^*_N/\mathbb{K}$, resp.) as left coidalgebras.
\end{proposition}
\begin{proof}
We start with the free sphere.

If such a quantum subgroup existed, then it would be of Kac type, and therefore its Haar idempotent $\Phi_\qk=h_\qk\circ \theta$ would be tracial. We will now show that the idempotent state associated to  $\mathscr{O}(S_+^{N-1})$ by Theorem \ref{thm-id} is not a trace.

Let $\mathbb{E}^\Phi_{bi}$ denote the conditional expectation in $\mathscr{O}(O_N^+)$ onto the *-subalgebra of $\mathscr{O}(O_N^+)$ generated by $u_{11}$, then we have $\Phi=\varepsilon \circ\mathbb{E}^\Phi_{l}= \varepsilon\circ\mathbb{E}^\Phi_{bi}$.

$\mathbb{E}^\Phi_{bi}$ is the orthogonal projection onto *-subalgebra genrated by $u_{11}$  for the inner product $\langle a,b\rangle = h(a^*b)$, and since we can compute the values of the Haar state on products of the algebraic generators $u_{11},\ldots,u_{NN}$ using the Weingarten calculus, we can compute  $\mathbb{E}^\Phi_{bi}$ and then $\Phi$. We find
\[
\mathbb{E}^\Phi_{bi} (u_{22}u_{11} u_{22}) = 0
\]
since $h_{O_{N}^+}(u_{11}^ku_{22}u_{11}u_{22})=0$ for all $k\in\mathbb{N}$ (there are no matching non-crossing pairings) and
\[
\mathbb{E}^\Phi_{bi} (u_{11}u_{22}^2) = u_{11} \mathbb{E}^\Phi_{bi} (u_{22}^2) =\frac{(N-2)u_{11}+u_{11}^3}{(N-1)^2}
\]
since
\[
h_{O_{N}^+}(u_{22}^2) = \frac{1}{N} \quad\mbox{ and } \quad h_{O_{N}^+}(u_{11}^2 u_{22}^2) =\frac{1}{N^2-1}
\]
If follows that
\[
\Phi(u_{11}u_{22}^2) = \varepsilon\left(\frac{(N-2)u_{11}+u_{11}^3}{(N-1)^2}\right)=\frac{1}{N-1} \not= 0 = \Phi(u_{22}u_{11}u_{22}).
\]

The case of the half-liberated sphere $S^{N-1}_*$ is similar.
Let us recall that a pairing is called "balanced," if each pair connects a black leg to a while leg, when we label its legs alternately black and white: $\bullet \circ\bullet \circ \cdots$. Denote the set of balanced pairings of $n$ elements by $P_2^*(n)$. The Weingarten formula for $O_N^*$ uses balanced pairings. The balanced pairings $P_2^*(4)$ and the non-crossing pairings $NC_2(4)$ of four elements are the same. Thus,  we get again the same values for Haar state in the  half-liberated case,
\begin{gather*}
h_{O_{N}^*}(u_{11}^ku_{22}u_{11}u_{22})=0, \qquad k=0,1,\ldots, \\
h_{O_{N}^*}(u_{22}^2) = \frac{1}{N} \quad\mbox{ and } \quad h_{O_{N}^*}(u_{11}^2 u_{22}^2) =\frac{1}{N^2-1},
\end{gather*}
and we get the same conclusion.
\end{proof}

We want to compute the eigenvalues and eigenspaces of $O_N^\times$-invariant Markov semigroups on $\mathscr{O}(S_\times^{N-1})$. First, we will give a decomposition of the Hilbert spaces $L^2((\mathscr{O}(S_\times^{N-1}), h)$, and $L^2((\mathscr{O}(O^\times_{N-1}), h)$ where $h$ denotes the Haar state on $C(O_N^\times)$, restricted to $C(S^{N-1}_\times)$. Set
\begin{gather*}
E_k=\text{span}\{u_{i_1j_1}\cdots u_{i_rj_r}:r\leq k\}\subset L_2(\mathscr{O}(O_N^\times),h);\\
H_k=\text{span}\{x_{i_1}\cdots x_{i_r}:r\leq k\}\subset L_2(\mathscr{O}(S_\times^{N-1}),h);\\
V_k=E_k \cap E_{k-1}^\perp; \quad D_k=H_k \cap H_{k-1}^\perp;\quad d_k=\dim D_k.
\end{gather*}
Then 
$L_2(\mathscr{O}(S_\times^{N-1}),h)= \bigoplus_{k\in \mathbb{N}} D_k$ and $L_2(\mathscr{O}(O^\times_{N}),h)= \bigoplus_{k\in \mathbb{N}} V_k$.  Furthermore $H_k=\mathbb{E}_{\ell}^\Phi(E_k)$, and thus $D_k=\mathbb{E}_{\ell}^\Phi(V_k)$.

Take a complete set $\{u^\pi:  \pi\in {\rm Irr}(O_N^\times) \}$ of mutually inequivalent, irreducible unitary representations. We know that the matrix $u=(u_{ij})$ is an irreducible unitary representation of $O_N^\times$ whose coefficients generate the function algebra. We can decompose its tensor powers $u^{\boxtimes s}=\bigoplus_{\pi\in I_s} n_\pi^s u^\pi$, where  $n_\pi^s$ denotes the multiplicity of $u^\pi$, and we used the notation $I_s:=\{\pi \in {\rm Irr}(O_N^\times): n_\pi^s\geq 1  \}$ . 

Then, for any $s\geq 2$ , we define 
$$u^{(s)}:=\bigoplus_{\pi \in J_s} u^\pi,  \quad \text{ where } J_s=I_s\backslash \cup_{0 \leq i\leq s-1}I_i.$$ 
In other words, $u^{(s)}$ is the direct sum of the ``new'' irreducible corepresentations in the decomposition of $u^{\boxtimes s}$, those which did not appear in the decompositions of $u^{\boxtimes i},  \forall i <s$. 

Since the linear space spanned by  coefficients of $\{u^{\boxtimes i}\}_{0\leq i\leq s}$ is $E_s$, by decomposition $E_s=\text{span}\{u_{pq}^\pi:  \pi \in I_i, 0\leq i\leq s \}$. Thus by definition, the linear space  spanned by  coefficients of $u^{(s)}$ is $V_s$. 

For the free case, by the fusion rule of $O_N^+$, we know that $I_s=\{s, s-2, s-4,\ldots\}$, therefore $J_s=\{s\}$. So $u^{(s)}$ is exactly the $s^{th}$ irreducible unitary corepresentation of $O_N^+$. But for other two cases, $u^{(s)}$ defined here may not be irreducible, but it is the direct sum of some mutually inequivalent irreducible unitary representations.

We state the argument above as a proposition:
\begin{proposition}
There exists a sequence of unitary corepresentations $(u^{(s)})_{s\in \mathbb{N}}$ of $O_N^\times$, such that the non-zero coefficients of $u^{(s)}$ are linearly independent and span $V_s$. In the free case, $u^{(s)}$ is irreducible.
\end{proposition}

The following lemma is the main step for characterising the idempotent state $\Phi$.
\begin{lemma}\label{phi x11=1}
There exists a basis for the Hilbert space $D_s$ associated to the corepresentation $u^{(s)}$, such that we get 
$$\Phi(s):=\left( \Phi(u_{jk}^{(s)})\right)_{1\leq j, k \leq d_s}=\delta_{j1}\delta_{i1}.$$
if we write $u^{(s)}=(u_{jk}^{(s)})_{1\leq j, k \leq d_s}$ w.r.t.\ to this basis. In other words, the corepresentation $u^{(s)}$ is unitarily equivalent to one for which applying $\Phi$ to it coefficient-wise produces a matrix with entry $1$ in the upper left corner and $0$ everywhere else.
\end{lemma} 
\begin{proof}
Since $\Phi$ is idempotent state, we can easily check that 
\[
\|\Phi_s\|\leq 1 \text{\ \ \ \ and \ \ \   } \Phi_s^2=\Phi_s ,
\]
which means that $\Phi_s$ is a projection in $B(D_s)$. We know that every projection matrix can be written as a diagonal matrix with coefficients $1$ and $0$ by choosing some suitable basis. 
So 
\[
\Phi_s=\left(
\begin{array}{cc}
\left. 
\begin{array}{ccc}
1&      &  \\
&\ddots&  \\
&      &1  \\
\end{array}  \right\}r_s   &
\begin{array}{ccc}
& \cdots & 0\\
\end{array}   \\
\begin{array}{c}
\vdots \\
0\\
\end{array}          & 
\begin{array}{ccc}
0 & &\\
&\ddots & \\
& &0
\end{array}
\end{array}\right).
\]
Denote the rank of this matrix by $r_s$. For all $k$, we take the basis of $D_k$ as above, so that for $0\leq i \leq r_k$, $\Phi (u_{ii}^{(s)})=1 $; otherwise $\Phi (u_{ij}^{(s)})=0$. Then for any $s \in \mathbb{N}$,
\begin{equation*}
\mathbb{E}_{bi}^\Phi(u_{ij}^{(s)})=\sum_{p,q}\Phi (u_{ip}^{(s)})u_{pq}^{(s)}\Phi (u_{qj}^{(s)})=
\begin{cases}
u_{ij}^{(s)} & \mbox{if $1\leq i,j \leq r_s $}\\
0 & \mbox{otherwise}.
\end{cases}
\end{equation*}
Moreover, the conditional expectation $\mathbb{E}_{bi} ^\Phi$ sends $E_s$ onto
\[
{\rm Pol}_s(u_{11}):=\{1,u_{11}, u_{11}^2, \cdots, u_{11}^s\}.
\]
Thus,
\[
\dim({\rm Pol}_s(u_{11}))=\dim(\mathbb{E}_{bi}^\Phi(E_s))=\dim(\mathbb{E}_{bi}^\Phi \left(\bigoplus_{k=0}^s(V_k)\right)) =\sum_{k=0}^s r_k^2,
\]
which implies
$$r_s^2=\dim({\rm Pol}_s(u_{11}))-\dim({\rm Pol}_{s-1}(u_{11}))=s+1-s=1.$$ 
\end{proof}

This theorem tells us that  $u^{(s)}_{11}=\mathbb{E}_{bi}^\Phi(u_{11}^{(s)}) \in {\rm Pol}_s(u_{11})$. Moreover, the algebra $\mbox{*-alg}\{u_{11}\}$ as a subalgebra of $\mathscr{O}(O_N^\times)$ and can be identified with the algebra of polynomials on the interval $[-1,1]$. Therefore, there exists $q_k^\times \in {\rm Pol}[-1, 1]$ such that $q^\times_k(u_{11})=u_{11}^{(k)}$. Since $h_{O_N^\times}(q_n^\times(u_{11})q_m^\times(u_{11}))=h_{O_N^\times}(u^{(n)}_{11}u^{(m)}_{11})=C\delta_{nm}$, $(q_s)_{s\in \mathbb{N}}$ is a family of orthogonal polynomials. The measure of orthogonality of these polynomials is the probability meeasure obtained by evaluating the spectral measure of $u_{11}$ in the Haar state. Since $u_{11}$ is hermitian and we have $\|u_{11}\|\le 1$, we get a measure that is supported on $[-1,1]$ (which explains why we consider only the values of our polynomials on this interval).

The restriction of the counit to $\mbox{*-alg}\{u_{11}\}$ corresponds to evaluation of a polynomial in the boundary point $1$, i.e. $\varepsilon(p(u_{11}))=p(1), \forall p\in {\rm Pol}[-1,1]$. Therefore we obtain the following result, in the same manner as in \cite[Proposition 10.1]{cipriani+franz+kula14}.

\begin{proposition}\label{value of psi} \cite[Proposition 10.1]{cipriani+franz+kula14}.
Let $\psi$ be a conditionally positive functional on ${\rm Pol}[u_{11}]$. Then there exist a unique pair $(b,\nu)$ consisting of a real number $b\ge 0$ and a finite measure $\nu$ on $[-1,1]$ such that 
\[
\psi(p) = -b p'(1) + \int_{-1}^1 \frac{p(x)-p(1)}{x-1} {\rm d}\nu(x)
\]
for any polynomial $p$. Conversely, every $\psi$ of this form is conditionally positive.
\end{proposition}

Applying the above proposition, we can compute the eigenvalues of Markov semigroups.

\begin{theorem}\label{eigenvalue for the markov semigroups on spheres}
For any $O_N^\times$-invariant strongly continuous Markov semigroup $(T_t)_{t\geq 0}$ on sphere $\mathscr{O}(S_\times^{N-1})$, there exists a pair  $(b,\nu)$, with $b$ a positive number and $\nu$  a finite measure on $[-1,1]$, such that the generator $L$ of $(T_t)_{t\ge 0}$ satisfying,
$$L(x_s)=\lambda_{s}x_s \qquad \forall x_s \in D_s, $$
where
\[
\lambda_s = - b(q_s^\times)'(1) + \int_{-1}^1 \frac{q_s^\times(x)-1}{x-1}{\rm d}\nu(x).
\]
Moreover, if $T_0=\mathrm{id}$, then for any $t\geq0$, $T_t(x_s)=e^{t\lambda_{s}}x_s,\forall x_s \in D_s.$
\end{theorem}

\begin{proof}
Theorem \ref{Theorem: strongly continous G invariant semigroups are differentiable} guarantees the existence of generator operator $L$ , and the Markov property makes $\psi:=\varepsilon\circ L$ conditionally positive. 

By Lemma \ref{phi x11=1}, we can compute $\mathbb{E}_{\ell}^\Phi (u^{(k)}_{ij})=\delta_{1i}u^{(k)}_{1j}$, which implies $D_s=\mathbb{E}_{\ell}^\Phi(V_k)=span \{u^{(s)}_{1j}: 1\leq j \leq d_k\}$. 
Then for any $u_{1j}^{(s)}\in D_s$,
\[
L(u_{1j}^{(s)})=\sum_k\psi(u_{1k}^{(s)})\mathbb{E}^\Phi_\ell(u_{kj}^{(s)})
=\psi(u_{11}^{(s)})u_{1j}^{(s)}.
\]
Now, we just need to consider $\psi|_{\mbox{*-alg}(u_{11})}$ which induces the pair  $(b,\nu)$ by Proposition \ref{value of psi}. By linearity of $L$, we can get the eigenvalues for $D_s$,
\[
\lambda_s =\psi(u_{11}^{(s)})= - b(q_s^\times)'(1) + \int_{-1}^1 \frac{q_s^\times(x)-1}{x-1}{\rm d}\nu(x),
\]
since $q_s^\times(1)=\varepsilon (u_{11}^{(s)})=1$.
\end{proof}

We point out here that we have three different families of orthogonal polynomials $\{q_s^\times(x)\}$ associated to $S_\times^{N-1}$, since the Haar states $h_{O^\times_N}$ depend on $\times\in\{\emptyset,*,+\}$. We will desccribe these orthogonal polynomials case by case.

\subsection{The classical sphere $S^{N-1}$}

Here, $(q_s(x))_{s\in \mathbb{N}}$ means the family of the orthogonal polynomials associated to classial sphere. It is well known that the distribution of $u_{11}^2$ for the classical sphere is the beta distribution with parameters $(1/2,(N-1)/2)$. In other words,
\[
h_{S^{N-1}}(\phi(u_{11}^2))=C\int_{0}^{1}\phi(t)\frac{1}{\sqrt{t}}(1-t)^{\frac{N-3}{2}}dt,
\]
where $C=\frac{\Gamma(\frac{N}{2})}{\Gamma(\frac{1}{2})\Gamma(\frac{N-1}{2})}$.
The integral vanishes on the odd polynomials, i.e. $h_{S^{N-1}}(u_{11}^{2k+1})=0$. Therefore $h_{S^{N-1}}\left(\frac{f(u_{11})+f(-u_{11})}{2}\right)=h_{S^{N-1}}(f(u_{11}))$.
\begin{align*}
h_{S^{N-1}}(f(u_{11}))&=h_{S^{N-1}}\left(\frac{f(u_{11})+f(-u_{11})}{2}\right)\\
&=C\int_{-1}^{1} f(t)(1-t)^{\frac{N-3}{2}}(1+t)^{\frac{N-3}{2}}{\rm d}t
\end{align*}

The spectral measure of $u_{11}$ is the probability measure on the interval $[-1, 1]$: $$\mu(dt)=C(1-t)^{\frac{N-3}{2}}(1+t)^{\frac{N-3}{2}}dt,$$ whose family of orthogonal polynomials is well known. Namely, we get the Jacobi polynomials (or ultraspherical polynomials) with parameters $\alpha=\beta=(N-3)/2$, which we will denote by $(J_s)_{s\in \mathbb{N}}$.

Recall that Jacobi polynomials are given by:
\[
J_s(x)=\sum_{r=0}^{s}\binom{s+(N-3)/2}{r}\binom{s+(N-3)/2}{(N-3)/2-r}\left(\frac{x-1}{2}\right)^{s-r}\left(\frac{x+1}{2}\right)^{r}.
\]
Their orthogonality relation is given by 
\[
\int_{-1}^{1}J_k(x)J_m(x)\mu(dt)=\delta_{km}\cdot C \frac{2^{N-2}}{2k+N-2}\frac{\Gamma(k+(N-1)/2)^2}{\Gamma(k+N-2)n!}.
\]
Moreover, they satisfy the differential equation
$$(1-x^2)J_s^{\prime\prime}(x)-x(N-1)J_s^\prime(x)=-s(s+N-2)J_s(x).$$
We need these polynomials in the form $q_s(x)=J_s(x)/J_s(1)$.

Therefore,
$$q_s^\prime(1)=\frac{s(s+N-2)}{N-1}.$$
We can relate our result to the \textit{Morkov sequence problem}. For a given orthonormal basis $\{f_0=\mathbf{1},f_1,f_2,\ldots\}$ of the $L^2$-space of some probability space, this problem of ask for the classification of all sequences $(\lambda_n)_{n\ge 0}$ such that $K(f_n)=\lambda_n f_n$ defines Markov operator, cf.\ \cite{bakry}. In \cite[Theorem 2]{bochner1954}, Bochner answered this problem for the Jacobi polynomials. Since we found that the Jacobi polynomials are the eigevectors for any $O_N$-invariant Markov semigroup on $S^{N-1}$, our Theorem \ref{eigenvalue for the markov semigroups on spheres} recovers \cite[Theorem 3]{bochner1954}. 

\subsection{The half-liberated sphere $S_*^{N-1}$} 

Next we consider the  half-liberated sphere $S^{N-1}_*$. 

Banica \cite[Propsition 6.6]{banica16} 
determined the law of $x_{i_1}\cdots x_{i_k}$ with respect to the Haar state $h_{S^{N-1}_*}=h_{O_N^*}|_{C(S^{N-1}_*)}$ (there is a small misprint in \cite[Propsition 6.6]{banica16}, which we correct below).

\begin{proposition}\label{law of x_i...x_k for half liberated}
The half-liberated integral of $x_{i_1}\cdots x_{i_k}$ vanishes, unless each index $a$ appears the same number of times at odd and even positions in $i_1,\ldots , i_k$. We have
\[
\int_{S_*^{N-1}}x_{i_1}\cdots x_{i_k}dx=\frac{(N-1)!\ell_1!\cdots\ell_n!}{(N+\sum \ell_i-1)!}
\]
where $\ell_i$ denotes this number of common occurrences of $i$ in the $k$-tuple $(i_1,\ldots,i_k)$. 
\end{proposition}

This proposition allows to describe the spectral distribution of $u_{11}=x_1$ w.r.t.\ the Haar state.

\begin{corollary}\label{distrubition of x11 for the semi-free case}
The distribution of $u_{11}$ in the half-liberated case is given by:
\[
h_{S^{N-1}_*}(f(u_{11}))=\int_{-1}^{1}f(t)\mu(dt),\qquad \forall f\in C([-1,1])
\]
where $\mu(dt)=(N-1)(1-t^2)^{N-2}|t|dt$.
\end{corollary}

\begin{proof}
This  proof repeats the arguments of \cite[Propsitions 6.5 and 6.6]{banica16}.

Let $C=\frac{2^N}{2N\pi^N}\cdot \Gamma(N+1)=(\frac{2}{\pi})^N(N-1)!$, then
\begin{align*}
h_{S^{N-1}_*}(x^{2k})&=\int_{S_{\mathbb{C}}^{N-1}}|z_1|^{2k}dz\\
&=C\int_{S_{\mathbb{R}}^{2N-1}}(x_1^2+y_1^2)^kd(x,y)\\
&=C\int_{0}^{\pi /2}\int_{0}^{\pi /2} (\cos^2\theta_1+\sin^2\theta_1 \cos^2\theta_2)^k\sin^{2N-2}\theta_1 sin^{2N-3}\theta_2 d\theta_1d\theta_2\\
&\qquad \cdot \int_{0}^{\pi /2}\sin^{2N-4}\theta_3 d\theta_{3}\cdots \int_{0}^{\pi /2}\sin \theta_{2N-2}d\theta_{2N-2}\cdot \int_{0}^{\pi /2}d\theta_{2N-1}.
\end{align*}
First we can calculate that
\begin{align*}
C'&=C\int_{0}^{\pi /2}\sin^{2N-4}\theta_3 d\theta_{3}\cdots \int_{0}^{\pi /2}\sin \theta_{2N-2}d\theta_{2N-2}\cdot \int_{0}^{\pi /2}d\theta_{2N-1}\\
&=\left(\frac{2}{\pi}\right)^N(N-1)!\cdot \left(\frac{\pi}{2}\right)^{N-1} \frac{(2N-4)!!}{(2N-3)!!}\frac{(2N-5)!!}{(2N-4)!!}\cdots\frac{1!!}{2!!}\\
&=\frac{4}{\pi}(N-1)
\end{align*}
where $m!!=(m-1)(m-3)\cdots 1$. Let $t=\sqrt{\cos^2\theta_1+\sin^2\theta_1 \cos^2\theta_2}$, $u=\cos \theta_1$, then
\begin{align*}
h_{S^{N-1}_*}(x^{2k})&=C'\int_{0}^{1}\int_{0}^{t}t^{2k}(1-t^2)^{N-2}\frac{t}{\sqrt{t^2-u^2}}dudt\\
&=2(N-1)\int_{0}^{1}t^{2k}\cdot(1-t^2)^{N-2}tdt.
\end{align*}
Since the odd moments of $u_{11}$ vanish, we have $h_{S^{N-1}_*}(f(u_{11}))=h_{S^{N-1}_*}\left(\frac{f(u_{11})+f(-u_{11})}{2}\right)$. and
\begin{align*}
h_{S^{N-1}_*}(f(u_{11}))&=2(N-1)\int_{0}^{1}\left(\frac{f(u_{11})+f(-u_{11})}{2}\right)\cdot(1-t^2)^{N-2}tdt\\
&=(N-1)\int_{-1}^{1}f(u_{11})(1-t^2)^{N-2}|t|dt.
\end{align*}
\end{proof}

Now we determine the family of orthogonal polynomials associated to the probability measure $\mu$ defined in Corollary \ref{distrubition of x11 for the semi-free case}.

The standard notation for hypergeometric functions is 
$${ }_rF_s\left({a_1, \cdots , a_r \atop b_1, \cdots, b_s}; x\right)=\sum_{n=0}^{\infty}\frac{(a_1)_n\cdots (a_r)_n}{(b_1)_n \cdots (b_s)_n}\frac{x^n}{n!}$$
where the shifted factorial $(a)_n$ is defined by
\begin{equation*}
(a)_n=
\begin{cases}
a(a+1)\cdots (a+n-1), & n=1,2,\cdots\\
1, & n=0. 
\end{cases}
\end{equation*}
They satisfy 
\begin{equation}\label{the derivative function of hypergeometric function}
\frac{d}{dx} { }_rF_s\left({a_1, \cdots , a_r \atop b_1, \cdots, b_s}; x\right)=\frac{\prod_{i=1}^ra_i}{\prod_{j=1}^sb_i}{ }_rF_s\left({a_1+1, \cdots , a_r+1 \atop b_1+1, \cdots, b_s+1}; x\right).
\end{equation}

And by Gauss' theorem we have
$${ }_2F_1\left({a, b \atop c }; 1\right)=\frac{\Gamma(c)\Gamma(c-a-b)}{\Gamma(c-a) \Gamma(c-b)}.$$

\begin{definition}
We define the family \emph{half-liberated spherical polynomials} (or ``*-polynomials'') by
\begin{align*}
P_{2k}(x)&=(-1)^k\left(\begin{array}{c}N+2k-2 \\ k \end{array}\right)^{-1}{ }_2F_1\left(\begin{array}{c} -k,\, N+k-1\\ 1 \end{array};\,x^2\right) \\
&=\sum_{r=0}^{k}(-1)^{k+r}\left(\begin{array}{c}k \\ r \end{array}\right)^2\left(\begin{array}{c} N+2k-2 \\ k-r\end{array}\right)^{-1} x^{2r}, \\
P_{2k+1}(x)&=x\cdot(-1)^k(k+1)\left(\begin{array}{c} N+2k-1 \\ k \end{array}\right)^{-1}{ }_2F_1\left(\begin{array}{c}-k,\, N+k\\ 2 \end{array} ;\,x^2\right) \\   
&=\sum_{r=0}^{k}(-1)^{k+r} \left(\begin{array}{c} k \\ r \end{array}\right)\binom{k+1}{r+1}\binom{N+2k-1}{k-r}^{-1}x^{2r+1}. 
\end{align*}
\end{definition} 

\begin{proposition}
The family of ``*-polynomials'' satisfies the following three-term recurrence relation:
\[
P_s(x)=xP_{s-1}(x)-\omega_{s-2}P_{s-2}(x) \qquad \forall s\geq 2,
\]
where $\omega_\ell=\frac{[(\ell+2)/2](N-1+[\ell/2])}{(N+\ell)(N+\ell-1)}$.
Moreover, the ``*-polynomials'' are the orthogonal polynomials for the probability measure $\mu(dt)=(N-1)(1-t^2)^{N-2}|t|dt$.
\end{proposition}

\begin{proof}
We can easily check that for any $k\geq 1$,
$$xP_{2k}(x)-\frac{k(N+k-2)}{(N+2k-2)(N+2k-1)}P_{2k-1}(x)=P_{2k+1}(x),$$
$$xP_{2k-1}(x)-\frac{k(N+k-2)}{(N+2k-2)(N+2k-3)}P_{2k-2}(x)=P_{2k}(x).$$
Therefore the three-term recurrence relation holds.

By the Proposition \ref{law of x_i...x_k for half liberated}, we can calculate
\begin{align*}
\int_{-1}^{1}P_{2k}(x)\mu(dx)&=\sum_{r=0}^{k}(-1)^{k+r}{k \choose r}^2{N+2k-2 \choose k-r}^{-1}\frac{(N-1)!r!}{(N+r-1)!}\\
&=(-1)^k\binom{N+2k-2}{k}^{-1}\sum_{r=0}^{k}\frac{(-1)^rk!(N+k+r-2)!(N-1)!}{(k-r)!(N+k-2)!(N+r-1)!r!};\\
&=(-1)^k\binom{N+2k-2}{k}^{-1}{ }_2F_1\left({-k, N+k-1 \atop N }; 1\right)\\
&=(-1)^k\binom{N+2k-2}{k}^{-1}\frac{\Gamma(N)}{\Gamma(1-k) \Gamma(N+k)}\\
&=\delta_{0k};                          
\end{align*}
and all of  the odd moments vanish, i.e., $\int_{-1}^{1}P_{2k+1}(x)\mu(dx)=0$.

We now prove the orthogonality by induction.

Clearly, $\forall n>0$, $\int_{-1}^{1}P_{n}(x)P_0(x)\mu(dx)=0.$

Assume that for any $0\leq k\leq s$, $\int_{-1}^{1}P_{n}(x)P_k(x)\mu(dx)=0$ holds for all $n>k$.
Then consider $s+1$, and $n>s+1$. Using the three-term recurrence relation, we get
\begin{align*}
\int_{-1}^{1}P_{n}(x)P_{s+1}(x)\mu(dx)&=\int_{-1}^{1}P_{n}(x)(xP_{s}(x)-\omega_{s-1}P_{s-1}(x))\mu(dx)\\ 
&=\int_{-1}^{1}xP_{n}(x)P_{s}(x)\mu(dx)+0\\
&=\int_{-1}^{1}(P_{n+1}(x)+\omega_{n-1}P_{n-1})P_{s}(x)\mu(dx)\\
&=0.
\end{align*}
Moreover, 
\begin{align*}
\int_{-1}^{1}P_{s}^2(x)\mu(dx)&=\int_{-1}^{1}P_s(x)(xP_{s-1}-\omega_{s-2}P_{s-2})\mu(dx)\\
&=\int_{-1}^{1}(P_{s+1}(x)+\omega_{s-1}P_{s-1}(x))P_{s-1}(x)\mu(dx)\\
&=\omega_{s-1}\int_{-1}^{1}P_{s-1}^2(x)\mu(dx)=\omega_0\omega_1\cdots\omega_{s-1},
\end{align*} so that
\[
\int_{-1}^{1}P_{m}(x)P_{n}(x)\mu(dx)=\omega_0\omega_1\cdots\omega_{n-1}\cdot\delta_{mn}.
\]
\end{proof}

\begin{remark}
We change the normalisation of these polynomial to get the sequence $q^*_s(x)=\frac{P_s(x)}{P_s(1)}$ which satisfies the conditions of Theorem \ref{eigenvalue for the markov semigroups on spheres}.

We have
\begin{align*}
P_{2k}(1)&=(-1)^k{N+2k-2 \choose k}^{-1}{ }_2F_1\left({-k, N+k-1\atop 1};1\right)\\
&=(-1)^k\frac{k!(N+k-2)!}{(N+2k-2)!}\frac{\Gamma(1)\Gamma(2-N)}{\Gamma(k+1)\Gamma(2-N-k)}\\
&=\frac{(N+k-2)!(N+k-2)!}{(N+2k-2)!(N-2)!}\\
P_{2k+1}(1)&=(-1)^k(k+1){N+2k-1 \choose k}^{-1}{ }_2F_1\left({-k, N+k\atop 2};1\right)\\
&=(-1)^k(k+1)\frac{k!(N+k-1)!}{(N+2k-1)!}\frac{\Gamma(2)\Gamma(2-N)}{\Gamma(k+2)\Gamma(2-N-k)}\\
&=\frac{(N+k-1)!(N+k-2)!}{(N+2k-1)!(N-2)!}          
\end{align*}
Therefore
$$q^*_{2k}(x)=(-1)^k{N+k-2 \choose k}^{-1}{ }_2F_1\left({-k, N+k-1\atop 1};x^2\right)$$

\[
q^*_{2k+1}(x)=x\cdot(-1)^k(k+1){N+k-2 \choose k}^{-1}{ }_2F_1\left({-k, N+k\atop 2};x^2\right).
\]
\end{remark}
The following formula gives the eigenvalues of the generator of the $O_N^*$-invariant semigroup on the half-liberated sphere $S_*^{N-1}$ associated to the pair $b=1$ and $\nu=0$. By analogy with the classical sphere, these values can be considered as the eigenvalues of the Laplace operator of the half-liberated sphere (up to a rescaling by $N-1$, see Remark \ref{rem-laplace}).
\begin{corollary}\label{the value of q_s^* '(1)}
For any $k\geq 0$,
\begin{align*}
(q_{2k}^{*})'(1)&=\frac{2k(N+k-1)}{N-1}\\
(q_{2k+1}^{*})'(1)&=\frac{(2k+1)N+2k^2-1}{N-1}.
\end{align*}
\end{corollary}

\begin{proof}
$q_0^{'}(1)=0$ is obvious.

For $k\geq 1$, by the equation \eqref{the derivative function of hypergeometric function}, we have
\begin{align*}
(q_{2k}^{*})'(1)&=\left.\frac{2x\frac{d}{dx^2}\ {   }_2F_1\left({-k, N+k-1\atop 1};x^2\right)}{{ }_2F_1\left({-k, N+k-1\atop 1};1\right)}\right|_{x=1}\\
&=\frac{-2k(N+k-1){ }_2F_1\left({-k+1, N+k\atop 2};1\right)}{{ }_2F_1\left({-k, N+k-1\atop 1};1\right)}\\
&=-2k(N+k-1)\frac{\Gamma(2)\Gamma(1-N)}{\Gamma(k+1)\Gamma(2-N-k)}\frac{\Gamma(k+1)\Gamma(2-N-k)}{\Gamma(1)\Gamma(2-N)}\\
&=\frac{2k(N+k-1)}{N-1};\\
(q_{2k+1}^{*})'(1)&=\left. \frac{\frac{d}{dx}\left(x\cdot {\ }_2F_1\left({-k, N+k\atop 2};x^2\right)\right)}{{ }_2F_1\left({-k, N+k\atop 2};1\right)}\right|_{x=1}\\
&=\left. \frac{{ }_2F_1\left({-k, N+k\atop 2};x^2\right)+x\cdot\left(2x\frac{d}{dx^2}{\ }_2F_1\left({-k, N+k\atop 2};x^2\right)\right)}{{ }_2F_1\left({-k, N+k\atop 2};1\right)}\right|_{x=1}\\    
&=\frac{{ }_2F_1\left({-k, N+k\atop 2};1\right)+2\cdot \frac{-k(N+k)}{2}{\ }_2F_1\left({-k+1, N+k+1\atop 3};1\right)}{{ }_2F_1\left({-k, N+k\atop 2};1\right)}\\
&=\frac{(2k+1)N+2k^2-1}{N-1}.        
\end{align*}
\end{proof}

\subsection{The free sphere $S_+^{N-1}$}

Finally, we consider about the free case.

In fact, due to the asymptotic semicircle law of $\sqrt{N+2}u_{11}$ when $N\rightarrow \infty$ \cite{banica+collins+zinn-justin09}, we expect that $q^+_s(x)\rightarrow U_s(\sqrt{N}x)/\sqrt{N^s}$, where $U_s(x)$ is the $s^{\rm th}$ Chebyshev polynomial of the second kind. Therefore, $\lim_{N\rightarrow \infty}q^+_s(x)=x^s$. So for the special case where the generating functional $\psi$ is associated to the pair $b=1$, $\nu=0$, the eigenvalues for the subspace $D_s$ converge as $N\to\infty$, $\lim_{N\rightarrow \infty}\lambda_s=-(x^s)'(1)=-s$. We now derive relations between polynomials $(q^+_s)_{s\ge 0}$ for general finite $N$.

\begin{proposition}\label{relation of q^+_s}
For any $N\in \mathbb{N}$, the orthogonal polynomials defined as above satisfy the following three-term recurrence relation:
\[
a_{s+1}q^+_{s+2}(x)=U_{s+1}(N)q^+_{s+1}(x)x-a_sq^+_s(x) \qquad \forall s\geq 0
\]
where $q_0^+(x)=1$, $q_1^+(x)=x$,
\[
a_s= \sum_{k=0}^{s}(-1)^{s+k}U_k(N)=
\begin{cases}
U_m(N)(U_m(N)-U_{m-1}(N))&\mbox{if } s=2m, \\
U_m(N)(U_{m+1}(N)-U_m(N))&\mbox{if } s=2m+1,
\end{cases}
\]
and where $U_s(N)$ denotes the value of the $s^{\rm th}$ Chebyshev polynomial of the second kind at the point $N$. 
\end{proposition}

\begin{proof}
For free orthogonal quantum group, the irreducible corepresentations have the following fusion rule \cite{banica92}:
\[
u^{(s+1)}\otimes u=u^{(s+2)}\oplus u^{(s)}.
\]
This implies that $u_{11}^{(s+1)}u_{11} \in V_{s+2} \oplus V_{s}$. Applying the two-sided conditional expectation  $\mathbb{E}_{bi}^\Phi$ to both sides, we see that $u_{11}^{(s+1)}\centerdot u_{11}$ can be written as the linear combination of $u_{11}^{(s+2)}$ and $u_{11}^{(s)}$. 

Let $\lambda_s$ be a number such that the coefficient of the highest degree of the polynomial  $\lambda_s q^+_s(x)$ is $1$. Since $q_s^+(1)=1$, we have
\[
\lambda_{s+2}q^+_{s+2}(x)=\lambda_{s+1}q^+_{s+1}(x)x-(\lambda_{s+1}-\lambda_{s+2})q^+_s(x).
\]
By the orthogonality of $\left(q^+_s(u_{11})\right)_{s\geq 0}$ and $h_{S^{N-1}_+}((q^+_s(u_{11}))^2)=h_{S^{N-1}_+}\left( \left( u_{11}^{(s)}\right)^2 \right) =1/U_s(N)$, we have

\begin{align*}
0&=h_{S^{N-1}_+}(\lambda_{s+2}q^+_{s+2}(u_{11})q^+_s(u_{11}))\\
&=h_{S^{N-1}_+}(\lambda_{s+1}q^+_{s+1}(u_{11})q^+_s(u_{11})u_{11})-(\lambda_{s+1}-\lambda_{s+2})h_{S^{N-1}_+}((q^+_s(u_{11}))^2)\\
&=\frac{ \lambda_{s+1}^2}{\lambda_s}h_{S^{N-1}_+}((q^+_{s+1}(u_{11}))^2)+0-(\lambda_{s+1}-\lambda_{s+2})h_{S^{N-1}_+}((q^+_s(u_{11}))^2)\\
&=\frac{ \lambda_{s+1}^2}{\lambda_sU_{s+1}(N)}-\frac{\lambda_{s+1}-\lambda_{s+2}}{U_s(N)}.
\end{align*}

Therefore 
\[
\frac{\lambda_{s+2}}{\lambda_{s+1}}=1-\frac{\lambda_{s+1}}{\lambda_{s}}\centerdot\frac{U_s(N)}{U_{s+1}(N)}.
\]
Set $a_s=\frac{\lambda_{s+1}}{\lambda_{s}}\cdot U_s(N)$, then 
\[
a_{s+1}q^+_{s+2}(x)=U_{s+1}(N)q^+_{s+1}(x)x-a_sq^+_s(x),
\]
and
\[
a_{s+1}=U_{s+1}(N)-a_s.
\]
From the latter equation we can  get 
\[
a_s=\sum_{k=0}^s(-1)^{s+k}U_k(N).
\]
\end{proof}

The following formula gives the eigenvalues of the generator of the $O_N^+$-invariant semigroup on the free sphere $S_+^{N-1}$ associated to the pair $b=1$ and $\nu=0$. By analogy with the classical sphere, these values can be considered as the eigenvalues of the Laplace operator of the free sphere (up to a rescaling by $N-1$, see Remark \ref{rem-laplace}).
\begin{corollary} \label{the value of q_s'}
\[
(q_s^+)'(1)=\sum_{r=0}^{s-1}\frac{\sum_{k=0}^rU_k(N)}{\sum_{k=0}^r(-1)^{r+k}U_k(N)}\qquad \forall s\geq 1
\]
\end{corollary}

\begin{proof}
Appling Proposition \ref{relation of q^+_s} and taking derivatives on both sides, we get 
$$a_{s+1}(q^+_{s+2})'(x)=U_{s+1}(N)\left( (q^+_{s+1})'(x)x+q^+_{s+1}(x)\right) -a_s(q^+_s)'(x).$$
Since $q_s^+(1)=1$, we have
$$a_{s+1}(q^+_{s+2})'(1)=U_{s+1}(N)(q^+_{s+1})'(1)+U_{s+1}(N)-a_s(q^+_s)'(1).$$
Rewrite this equation using $U_{s+1}(N)=a_{s+1}+a_s$,
\[
a_{s+1}\left((q^+_{s+2})'(1)-(q^+_{s+1})'(1)\right)=U_{s+1}(N)+a_s\left((q^+_{s+1})'(1)-(q^+_{s})'(1)\right).
\]
Therefore,
$$a_s\left((q^+_{s+1})'(1)-(q^+_{s})'(1)\right)=\sum_{k=0}^sU_k(N).$$
This implies
\[
(q^+_{s+1})'(1)=\sum_{r=0}^s\frac{\sum_{k=0}^rU_k(N)}{a_r}
\]
\end{proof}

We can get an estimate of these eigenvalues that grows linearly in $s$.
\begin{corollary}\label{coro-lambda-asymp}
For any $N\geq 2$,
\[
s\leq (q_s^+)'(1)\leq \frac{N+2}{N-2} s, \qquad \forall s\geq 0,
\]
(where the upper becomes $+\infty$ for $N=2$).
\end{corollary}

\begin{proof}
Using the relation $U_s(N)N=U_{s+1}(N)+U_{s-1}(N)$, we have
\[
\sum_{k=0}^{m}U_{2k}(N)=\frac{1}{2}\left(U_{2m}(N)+U_{0}(N)\right)+\frac{N}{2}\left(\sum_{k=1}^m U_{2k-1}(N)\right),
\]
and
\[
\sum_{k=0}^{m}U_{2k+1}(N)=\frac{1}{2}\left(U_{2m+1}(N)+U_{1}(N)\right)+U_0(N)+\frac{N}{2}\left(\sum_{k=1}^m U_{2k}(N)\right).
\]
Therefore
\[
\frac{\sum_{k=0}^rU_k(N)}{\sum_{k=0}^r(-1)^{r+k}U_k(N)}\leq \frac{N/2+1}{N/2-1}.
\]
\end{proof}

\begin{remark}\label{rem-laplace}
For the classical sphere, we know that the Laplace operator is the operator whose eigenvector are the Jacobi polynomials $J_s$ and whose eigenvalues are $\lambda_{s}=s(s+N-2)=-(N-1)q_s^\prime(1)$. So the generator for classical spheres in Theorem \ref{eigenvalue for the markov semigroups on spheres}, is induced from the generating functional $\psi$ associated to the pair $(b,\nu)=(N-1,0)$ is the Laplace operator. In the same manner, we may define the Laplace operator $\Delta_*$ on the half-liberated sphere and the Laplace operator $\Delta_+$ on free sphere.
\end{remark}

\begin{remark}\label{rem-central}
Recall that we showed in Proposition \ref{prop-central} that central convolution semigroups of states on $C^u(\mathbb{G})$ also induce $\mathbb{G}$-invariant Markov semigroups on any quantum space $\mathbb{X}$ equipped with a right $\mathbb{G}$-action. The generating functionals of central convolution semigroups of states on $C^u(O^+_N)$ were classified in \cite[Corollary 10.3]{cipriani+franz+kula14}. This gives the formula
\begin{equation}\label{eq-central}
\lambda_s = -b\frac{U'_s(N)}{U_s(N)} + \int_{-N}^N \frac{U_s(x)-U_s(N)}{U_s(N)(N-x)}\nu({\rm d}x),\qquad s=0,1,\ldots
\end{equation}
with $b$ a positive real number, $\nu$ a finite positive measure on the interval $[-N,N]$ and $(U_s)_{s=0}^\infty$ the Chebyshev polynomials of the second kind defined by $U_0(x)=1$, $U_1(x)=x$, $U_{s+1}(x) = xU_s(x)-U_{s-1}(x)$ for $s\ge 1$.

Recall again that by \cite[Theorem 5.3]{banica+collins+zinn-justin09} the distribution of $\sqrt{N+2}\,u_{11}$ converges uniformally to the semicircle distribution, which is the measure of orthogonality of the Chebyshev polynomials. This suggests that the eigenvalues given by Theorem \ref{eigenvalue for the markov semigroups on spheres} and in Equation \eqref{eq-central} for the free sphere $S_+^{N-1}$ should be close for large $N$.
\end{remark}

\subsection{Spectral dimensions}\label{subsec-specdim}
The Weyl formula for the eigenvalues of the Laplace-Beltrami operator $\Delta_\mathcal{M}$ on a compact Riemannian $C^\infty$-manifold $(\mathcal{M},g)$ of dimension $N$ states that
\[
N(\lambda) \sim_{\lambda\to+\infty} \frac{\lambda^{N/2}|\mathcal{M}| }{(4\pi)^{N/2} \Gamma\left(\frac{N}{2}+1\right)}
\]
cf.\ \cite{mp49}, where $|\mathcal{M}|$ denotes the volume of $(\mathcal{M},g)$, $N(\lambda)$ denotes the number eigenvalues of the Laplace-Beltrami operator that are less then or equal to $\lambda$, and $f\sim_{\lambda\to+\infty} g$ stands for ``asymptotically equivalent,'' i.e., for $\lim_{\lambda\to\infty} \frac{f(\lambda)}{g(\lambda)} = 1$. This implies that the zeta-function $\zeta_M(z) = \sum_{\lambda\in\sigma(\Delta_M)} m_\lambda \lambda^z$, where $m_\lambda$ denotes the multiplicity of the eigenvalue $\lambda$, has a simple pole in $\frac{N}{2}$, and that this value is also the abscissa of convergence of the series. For this reason, we define the ``spectral dimension'' $d_L$ of the spheres $S^{N-1}_*$ (w.r.t.\ a generator $L$) as the abscissa of convergence of the series $\sum_{s=0}^\infty m_s \lambda_s^{-z/2}$, where $(\lambda_s)_{s\ge 0}$ are the eigenvalues of $L$ which we classified in Theorem \ref{eigenvalue for the markov semigroups on spheres}. Note that this definition is equivalent to Connes' definition in \cite{connes04a,connes04b}, if we construct a Dirac operator $D_L$ from $L$ as in \cite{cipriani+franz+kula14}, since the eigenvalues of $D_L$ will be $(\pm\sqrt{\lambda_s})_{s\ge0}$

The spectral dimension $d_L$ is equal to the infimum of all $d>0$ such that the sum $\sum_s m_s(-\lambda_s)^{-d/2}$ is finite.

For simplicity, we will only consider the special case $b=1$ and $\nu=0$ of the eigenvalues given in Theorem \ref{eigenvalue for the markov semigroups on spheres}.

\subsubsection{The classical sphere $S^{N-1}$}

By definition of $D_s$,
\[
\dim D_s=\dim H_s -\dim H_{s-1},
\]
where
\[
H_s=\text{span}\{x_1^{k_1}\cdots x_{N}^{k_N}:k_1+\cdots + k_N \le s\}
\]

Since $x_1^2=1-\sum_{i=2}^N x_i^2$, we only need consider $k_1=0 $ or $k_1=1$ in above formula.

Recall that $|\{(k_1, k_2, \cdots, k_n) \in \mathbb{N}^n: k_1+k_2+\cdots k_n=M\}|=\binom{M+n-1}{n-1}$.

For $k_1=0$, 
$$\dim \text{span}\{x_2^{k_2}\cdots x_N^{k_N}: k_2+\cdots+k_N=s \}=\binom{s+N-2}{N-2};$$
and for $k_1=1$,
$$\dim \text{span}\{x_1x_2^{k_2}\cdots x_N^{k_N}: k_2+\cdots+k_N=s-1 \}=\binom{s+N-3}{N-2}$$
therefore,
\[
m_s=\dim D_s=\binom{s+N-2}{N-2}+\binom{s+N-3}{N-2}\asymp s^{N-2}.
\]
where the notation $a_s\asymp b_s$ for two sequences of strictly positive numbers means that they are of the same order of magnitude. More precisely, $a_s\asymp b_s$ means that there exist constants $c,C>0$ such that for all $s\in\mathbb{N}$, $c a_s\le b_s \le C a_s$.

For the eigenvalues we have $-\lambda_s=\frac{s(s+N-2)}{N-1}\asymp s^{2}$, and therefore we find $d_L=N-1$, as expected.

\subsubsection{The half-liberated sphere $S^{N-1}_*$}

Again, $\dim D_s=\dim H_s -\dim H_{s-1}$. Consider first the even case, i.e. $s=2m$.

Let $X=x_{\ell_1}x_{\ell_2}\cdots x_{\ell_{2m-1}} x_{\ell_{2m}}\in D_{2m}$. Use black dots ``$\bullet$'' for odd positions and white dots ``$\circ$'' for even positions, i.e., associate the diagram
\[
\bullet \circ \bullet \circ \cdots \bullet \circ.
\]
to the monomial $X$.
Since we have the relation $x_{\ell_1}x_{\ell_2}x_{\ell_3}=x_{\ell_3}x_{\ell_2}x_{\ell_1}$ for the generators, we can freely permute the generators $x_{\ell_{2k-1}}$ that are placed on black dots ``$\bullet$'' (i.e., in odd positions) among each other. Similarly, generators $x_{\ell_{2k}}$ sitting on white dots ``$\circ$'' (i.e., in even positions) can be permuted among each other..

Write $x_{\ell_{2k-1}}=a_{i_k}$ and $x_{\ell_{2k}}=b_{j_k}$, respectively, for the generators on black and white dots, then $X=a_{i_1}b_{j_1}\cdots a_{i_m}b_{j_m}$.

Since $a_{i_k}$ is commute among each other, we set $a_{i_1}a_{i_2}\cdots a_{i_m}=x_1^{p_1}x_2^{p_2}\cdots x_N^{p_N}$ with $p_1+p_2+\cdots p_N=m$. 
Similary, set $b_{i_1}b_{i_2}\cdots b_{i_m}=x_1^{q_1}x_2^{q_2}\cdots x_N^{q_N}$ with $q_1+q_2+\cdots q_N=m$.

Since $x_1^2=1-\sum_{i=2}^N x_i^2$, we can assume $p_1=0$ or  $q_1=0$. Indeed, if both monomials $a_{i_1}a_{i_2}\cdots a_{i_m}$ and  $b_{i_1}b_{i_2}\cdots b_{i_m}$ contain the generator $x_1$, then we can we could move $x_1$ to the first position in both the subwords $a_{i_1}a_{i_2}\cdots a_{i_m}$ and  $b_{i_1}b_{i_2}\cdots b_{i_m}$, and replace the resulting $x_1^2$ by $1-\sum_{i=2}^N x_i^2$. In this way get one monomial that is in $H_{s-2}$, and in the remaining terms the powers of $x_1$ in both subwords are reduced by $1$. Iterating this procedure we can express $X$ as a linear combination of monomials which have $p_1=0$ or  $q_1=0$.

Therefore,
 \begin{gather*}
 	\dim D_s = \dim\text{span}\left\{x_2^{p_2}\cdots x_N^{p_N}: \sum_{k=2}^N p_k=m\right\}  \cdot \dim\text{span}\left\{x_2^{q_2}\cdots x_N^{q_N}: \sum_{k=2}^N q_k=m\right\}\\
 	+\dim\text{span}\left\{x_2^{p_2}\cdots x_N^{p_N}: \sum_{k=2}^N p_k=m\right\}  \cdot \dim\text{span}\left\{x_1^{q_1}x_2^{q_2}\cdots x_N^{q_N}: q_1>0, \sum_{k=1}^N q_k=m\right\}\\
 	+\dim\text{span}\left\{x_1^{p_1}x_2^{p_2}\cdots x_N^{p_N}:p_1>0,\sum_{k=1}^N p_k=m\right\}  \cdot \dim\text{span}\left\{x_2^{q_2}\cdots x_N^{q_N}:\sum_{k=2}^N q_k=m\right\}\\
 	=\binom{m+N-2}{N-2}^2+2 \binom{m+N-2}{N-2}\binom{m+N-2}{N-1}\asymp m^{2N-3}\asymp s^{2N-3}.
 \end{gather*} 

Similary, when $s=2m+1$,
\begin{align*}
	\dim D_s&=\binom{m+N-2}{N-2}\binom{m+N-1}{N-2}+\binom{m+N-1}{N-2}\binom{m+N-1}{N-1}\\
	&\qquad+\binom{m+N}{N-1}\binom{m+N-2}{N-2}\\
	&\asymp s^{2N-3}.
\end{align*}

On the other hand, by Corollary~\ref{the value of q_s^* '(1)},
$-\lambda_s\asymp s^{2}$. Hence,
\[
d_L=2(N-1).
\]

Banica showed in \cite[Theorem 1.14]{banica16}. that $C(S^{N-1}_*)$ can be embedded into the C$^*$-algebra $M_2\big(C(S^{N-1}_{\mathbb{C}})\big)$ of continuous functions with values in $2\times2$-matrices on the complex sphere $S^{N-1}_\mathbb{C}=\{z=(z_1,\ldots,z_N)\in\mathbb{C}^N:\sum_{i=1}^N |z_i|^2=1\}$.
This embedding sends the generators $x_i$, $i=1,\ldots,N$, to the functions $\pi(x_i):S^{N-1}_{\mathbb{C}}\ni z=(z_1,\ldots,z_N)\mapsto \left(\begin{array}{cc} 0 & z_i \\ \overline{z}_i & 0 \end{array}\right)$.
Evaluating these functions in a point $z\in S^{N-1}_{\mathbb{C}}$ defines a unique 2-dimensional representation $\pi_z:C(S^{N-1}_{\mathbb{C}})\to M_2(\mathbb{C})$.
Two of these 2-dimensional representations $\pi_{z}$ and $\pi_{w}$, $z,w\in S^{N-1}_{\mathbb{C}}$, are unitarily equivalent if and only if there exists a complex number $\lambda$ with $|\lambda|=1$ such that $z=\lambda w$. This means that the embedding passes to the projective complexe sphere $P^{N-1}_{\mathbb{C}}=S^{N-1}_{\mathbb{C}}/\sim$, where $\sim$ is the equivalence relation on $S^{N-1}_{\mathbb{C}}$ defined by
\[
z_1\sim z_2 \Leftrightarrow \exists \lambda\in\mathbb{C}, z_1=\lambda z_2.
\]
Since the dimension of $P^{N-1}_{\mathbb{C}}$ as a real manifold is $2(N-1)$, this provides a heuristic explanation for the value of the spectral dimension $d_L$ for the half-liberated sphere $S^{N-1}_*$.

\subsubsection{The free sphere $S^{N-1}_+$}

For the free case, $D_s=\text{span}\{u_{1i}^{(s)}:1\leq  i\leq d_s\}$ where $u_{ij}^{(s)}$ are the cofficients of the $s^{\text{th}}$ irreducible corepresentation $u^{(s)}$, which has dimension $d_s=U_s(N)$. 

Let us first consider the case $N=2$. Since $U_s(2)=s+1$, we get $a_s=[s/2]+1$ and $\sum_{k=0}^r=\frac{(r+1)(r+2)}{2}$. By Corollary 7.14, we have
\begin{eqnarray*}
-\lambda_{2k+1} &=& (q_{2k+1}^+)\prime(1)=2k^2+4k+1, \\
-\lambda_{2k} &=& (q_{2k}^+)\prime(1)=2k^2+2k.
\end{eqnarray*}
Therefore,
\[
-\lambda_s \asymp s^2, \qquad m_s\asymp s.
\]
This implies $d_L=2$ for $N=2$. For $N=2$, the defining relation of the free sphere $S^2_+$ can be written as $x_2^2=1-x_1^2$, which implies $x_1x_2^2=x_2^2 x_1$, as well as the other half-commutation relations $x_i x_j x_k=x_kx_jx_i$, $i,j,k\in\{1,2\}$. So we have $C^u(S^2_+)\cong C^u(S^2_*)$, i.e., the free and the half-liberated two-dimensional spheres coincide.

By Corollary \ref{coro-lambda-asymp}, $\lambda_s\asymp s$ for $N\ge 3$. Furthermore, in this case $m_s=U_s(N)\asymp N^s$. Hence,
\[
d_L = \left\{\begin{array}{cl}
2 & \mbox{ if } N=2, \\
+\infty & \mbox{ if } N\ge 3.
\end{array}\right.
\]

This resembles the computation in \cite[Remark 10.4]{cipriani+franz+kula14}, where we found
\[
d_D = \left\{\begin{array}{cl}
3 & \mbox{ if } N=2, \\
+\infty & \mbox{ if } N\ge 3.
\end{array}\right.
\]
for the spectral dimension of a spectral triple constructed from a central generating functional on the free orthogonal quantum group $O_N^+$. 

%%%%%%%%%%
\section*{Acknowledgements}
%%%%%%%%%%

We thank Teo Banica, Adam Skalski, and Haonan Zhang for fruitful discussions and useful suggestions.

UF and XW were supported by the French ``Investissements d'Avenir'' program, project  ISITE-BFC (contract ANR-15-IDEX-03). XW was supported by the China Scholarship Council. We also acknowledge support by the French MAEDI and MENESR and by the Polish MNiSW through the Polonium programme.

\end{document}